\documentclass[12pt]{amsart}
\usepackage[margin=1in]{geometry}
\usepackage[utf8]{inputenc}
\usepackage{amsfonts,amsmath,amsthm,amssymb}
\usepackage[unicode]{hyperref}
\usepackage{listings}
\usepackage{lastpage}
\usepackage{tikz}
\usepackage[justification=centering]{caption}
\usetikzlibrary{cd}
\usepackage[mathscr]{euscript}
\usepackage{xcolor}
\usepackage{graphicx, subcaption}
\usepackage{mathabx}
\usepackage{enumitem}

\graphicspath{{./pics/}}	

\newcommand{\C}{{\mathbb C}}
\newcommand{\Isol}{\mathrm{Isol}}

\newcommand{\I}{{\mathbb I}}
\newcommand{\D}{{\mathbb D}}

\newcommand{\R}{{\mathbb R}}

\newcommand{\N}{{\mathbb N}}
\renewcommand{\D}{{\mathbb D}}

\newcommand{\M}{\mathcal{M}}

\newcommand{\Homeo}{\operatorname{Homeo}}

\newcommand{\id}{\operatorname{id}}

\renewcommand{\L}{\mathcal{L}}

\renewcommand{\mod}{\mathrm{mod}}
\newcommand{\dto}{\dashrightarrow}
\newcommand{\rto}{\righttoleftarrow}
\newcommand{\T}{\mathcal{T}}

\newcommand{\hide}[1]{}

\newcounter{main}

\theoremstyle{plain}
        \newtheorem{theorem}{Theorem}[section]
        \newtheorem*{theorem*}{Theorem}
        \newtheorem*{conj*}{Conjecture}
        \newtheorem{lemma}[theorem]{Lemma}
        \newtheorem{corollary}[theorem]{Corollary}
        \newtheorem{proposition}[theorem]{Proposition}
        \newtheorem{maintheorem}[main]{Main Theorem}        

\theoremstyle{definition}
        \newtheorem{definition}[theorem]{Definition}
        \newtheorem*{definition*}{Definition}

\theoremstyle{remark}
        \newtheorem{remark}[theorem]{Remark}

        \newtheorem{example}[theorem]{Example}

        \newtheorem*{example*}{Example}
        \newtheorem*{examples*}{Examples}        
        \newtheorem*{claim}{Claim}
        \newtheorem*{claim1}{Claim 1}

\newenvironment{subproof}[1][\proofname]{%
  \begin{proof}[#1]%
}{%
  \end{proof}%
}

\title[Finite and infinite degree Thurston maps with extra marked points]{Finite and infinite degree Thurston maps with extra marked points}

\author{Nikolai Prochorov}
\address {Aix-Marseille Université, CNRS, Institut de Mathématiques de Marseille, 13003 Marseille, France}
\email{nikolai.prochorov@etu.univ-amu.fr, prochorov41@gmail.com}

\date{\today}


\begin{document}

\begin{abstract}
    We investigate the family of marked Thurston maps that are defined everywhere on the topological sphere $S^2$, potentially excluding at most countable closed set of essential singularities. We show that when an unmarked Thurston map $f$ is realized by a postsingularly finite holomorphic map, the marked Thurston map $(f, A)$, where $A \subset S^2$ is the corresponding finite marked set, admits such a realization if and only if it has no degenerate Levy cycle. To obtain this result, we analyze the associated pullback map $\sigma_{f, A}$ defined on the Teichm\"uller space~$\T_A$ and demonstrate that some of its iterates admit well-behaved invariant complex sub-manifolds within $\T_A$. By applying powerful machinery of one-dimensional complex dynamics and hyperbolic geometry, we gain a clear understanding of the behavior of the map $\sigma_{f, A}$ restricted to the corresponding invariant subset of the Teichm\"uller space~$\T_A$.
\end{abstract}

\maketitle

\vspace{-0.4cm}

\tableofcontents

\newpage

\section{Introduction}

\subsection{Thurston theory for finite and infinite degree maps}\label{subsec: inro thurston theory}
In the one-dimensional rational dynamics, the crucial role is played by the family of \textit{postcritically finite} (or \textit{pcf} in short) rational maps, i.e., maps with all critical points being periodic or pre-periodic. In this context, one of the most influential ideas has been to abstract from the rigid underlying complex structure and consider the more general setup of postcritically finite \emph{branched self-coverings} of the topological $2$-sphere $S^2$. Nowadays, orientation-preserving pcf branched covering maps $f\colon S^2 \to S^2$ of topological degree $\deg(f)\geq 2$ are called \emph{Thurston maps} (\textit{of finite degree}), in honor of William Thurston, who introduced them to deepen the understanding of the dynamics of postcritically finite rational maps~on~$\widehat{\C}$.

These ideas can be extended to the transcendental setting to explore the dynamics of holomorphic maps of infinite degree that are not defined everywhere on the Riemann sphere. In this paper, we focus on \textit{postsingularly finite} (\textit{psf} in short) maps that are holomorphic everywhere on~$\widehat{\C}$ except at the closed at most countable closed set of \textit{essential singularities}. Here, the map is called postsingularly finite if it has finitely many \textit{singular values}, and each of them eventually becomes periodic or lands on an essential singularity under the iteration. We can generalize the notion of a Thurston map to postsingularly finite \textit{topologically holomorphic} non-injective maps $f \colon S^2 - E \to S^2$, where $E \subset S^2$ is at most countable closed set, $f$ does not extend continuously to a neighborhood of any of the points of $E$ and meets a technical condition of being an \textit{admissible type} map; see Sections \ref{subsec: topologically holomorphic maps} and~\ref{subsec: thurston maps}. It is important to note  that the set $E$ is non-empty if and only if $f$ is \textit{transcendental}, meaning that it has infinite topological degree. For simplicity, we will use the notation $f \colon S^2 \dto S^2$ to indicate that the Thurston map $f$, whether finite or infinite degree, might not be defined on at most countable closed set of $S^2$.

For a Thurston map $f \colon S^2 \dto S^2$, the \textit{postsingular set} $P_f$ is defined as the union of all orbits of the singular values of the map $f$. It is important to note that some of these orbits might terminate after several iterations if a singular value reaches a point where the map~$f$ is not defined. The elements of the postsingular set $P_f$ are called the \textit{postsingular values} of the Thurston map $f$. If the map $f$ is defined on the entire sphere $S^2$, we simply refer to its \textit{postcritical set} and \textit{postcritical values}, as the set of singular values of $f$ coincides with the set of its \textit{critical values}.  Two Thurston maps are called \emph{combinatorially} (or \emph{Thurston}) \emph{equivalent} if they are conjugate up to isotopy relative to their postsingular~sets; see Definition \ref{def: comb equiv}.

A fundamental question in this context is whether a given Thurston map $f \colon S^2 \dto S^2$ can be \emph{realized} by a psf holomorphic map with the same combinatorics, that is, if $f \colon S^2 \dto S^2$ is combinatorially equivalent to a psf holomorphic map.  If the Thurston map $f$ is not realized, then we say that $f$ is \textit{obstructed}. William Thurston answered this question for Thurston maps of finite degree in his celebrated \emph{characterization of rational maps}: 
if a finite degree Thurston map $f \colon S^2 \to S^2$ is not a \textit{$(2, 2, 2, 2)$-map}, then $f$ is realized by a pcf rational map if and only if $f$ has no \emph{Thurston obstruction} \cite{DH_Th_char}. Such an obstruction is given by a finite collection of disjoint simple closed curves in $S^2 - P_f$ with certain invariance properties under the map $f$. In many instances, it suffices to restrict to simpler types of Thurston obstructions provided by \textit{Levy cycles}; see \cite[Theorem 10.3.8]{Hubbard_Book_2}, \cite[Proposition 1.1]{sharland}, \cite[Corollary~1.5]{critically_fixed}, or \cite[Theorems 7.6 and 8.6]{park} for examples of such cases. The family of $(2, 2, 2, 2)$-maps that are not covered by Thurston's characterization result consists of all Thurston maps with four postcritical values, where each critical point is \textit{simple} and strictly pre-periodic. 

The same characterization question can be also asked in the transcendental setting. The first breakthrough in this area was obtained in \cite{HSS}, where it was shown that an \textit{exponential} Thurston map is realized if and only if it has no Levy cycle. More recently, the characterization question was resolved for a broad class of Thurston maps with four postsingular values and a single essential singularity in \cite{small_postsingular_set}. Furthermore, the results of \cite{Shemyakov_Thesis} and \cite{our_approx} suggest that a Thurston-like criterion for realizability may hold in a greater generality. However, the characterization question in the transcendental setting remains largely open, as many of the techniques used in Thurston theory for finite degree maps do not extend to this context.

Let $f \colon S^2 \dto S^2$ be a Thurston map and $A$ be a subset of $S^2$. We say that $A$ is \textit{$f$-pseudo-invariant} if the orbit of each point $a \in A$, whether it terminates after several iterations or not, remains within $A$. Clearly, every $f$-invariant set is $f$-pseudo-invariant; however, the converse is not necessarily true, since one of the points in 
$A$ could be an essential singularity of $f$. If $A$ is a finite $f$-pseudo-invariant subset of $S^2$ containing $S_f$ (and therefore, the entire postsingular set $P_f$), we can consider the pair $(f, A)$ called a \textit{marked Thurston map}, where $A$ is referred to as the \textit{marked set}. For clarity in notation, we denote the marked Thurston map $(f, A)$ by $f \colon (S^2, A) \rto$. If no marked set is explicitly mentioned, it is assumed to coincide with the postsingular set $P_f$. Two marked Thurston maps $f_1 \colon (S^2, A_1) \rto$ and $f_2 \colon (S^2, A_2) \rto$ are said to be \textit{combinatorially equivalent} if they are conjugate up to isotopy relative to their marked sets; see Definition \ref{def: comb equiv}. It implies, for instance, that $f_1|A_1$ and $f_2|A_2$ are conjugate. Furthermore, for the unmarked Thurston map $f \colon S^2 \dto S^2$, we can explore whether $f$ is realized with respect to different marked sets. Naturally, the larger the marked set we consider, the finer the notion of combinatorial equivalence becomes. Finally, it is worth noting that the notion of a marked Thurston map naturally emerges even if we focus solely on unmarked Thurston maps. This phenomenon, for instance, arises in \textit{Pilgrim's decomposition theory}; see \cite{Pilgrim_Comb}, \cite{Selinger_Top_Obstr}, and \cite{overview_bartholdi_dudko}.

This leads to the same realizability question as before, but now in the broader context of marked Thurston map. According to \cite{Buff} (see also \cite{MargalitWinarski} for an alternative approach), Thurston's characterization criterion extends directly to the setting of finite degree marked Thurston maps that are not $(2, 2, 2, 2)$-maps. Marked Thurston $(2, 2, 2, 2)$-maps were studied in \cite{Selinger_Yampolski} and \cite{Bartholdi_Dudko}. Additionally, although not explicitly stated, the result of \cite{HSS} also applies in the context of marked exponential Thurston maps. Finally, in \cite[Theorem~4.4 and Example~4.15]{small_postsingular_set}, it is shown that a marked Thurston maps with at most one essential singularity, at most three postsingular values, and four marked points is realized if and only if it has no Levy cycle.

\subsection{Pullback maps}\label{subsec: intro pullback maps}
From this point forward, we assume that any marked set contains at least three points, as the case of fewer marked points is rather trivial; see Proposition \ref{prop: small set}.
The key method in determining whether a given finite or infinite degree marked Thurston map $f \colon (S^2, A) \rto$ is realized by a psf holomorphic map is the analysis of the dynamics of an operator~$\sigma_{f, A}$, known as the \textit{pullback map}, defined on the \textit{Teichm\"uller space}~$\T_A$; see Sections~\ref{subsec: teichmuller and moduli spaces} and~\ref{subsec: pullback maps}. Roughly speaking, the Teichm\"uller space $\T_A$ encodes all possible ways to put a complex structure, which is biholomorphic to a punctured Riemann sphere, on the punctured topological sphere $S^2 - A$, up to a certain equivalence depending on the set~$A$. Crucially, the Thurston map $f\colon (S^2, A) \rto$ is realized if and only if the pullback map $\sigma_{f, A}$ has a fixed point in $\T_A$; see Proposition \ref{prop: fixed point of sigma}. Moreover, the dynamics of the pullback map encodes many other properties of the corresponding Thurston map; see, for instance, \cite{pullback_map}, \cite{Pullback_invariants}, and \cite{correpondences}.

Teichm\"uller space has the structure of $(|A| - 3)$-complex-dimensional manifold and can be equipped with a complete metric $d_T$, known as the \textit{Teichm\"uller metric}. It is known that every pullback map $\sigma_{f, A}$ is holomorphic and \textit{1-Lipschitz} on $\T_A$, meaning $d_T(\sigma_{f, A}(\tau_1), \sigma_{f, A}(\tau_2)) \leq d_T(\tau_1, \tau_2)$ for all $\tau_1, \tau_2 \in \T_A$. Moreover, if $f$ is transcendental, then $\sigma_{f, A}$ is \textit{distance-decreasing}, i.e., $d_T(\sigma_{f, A}(\tau_1), \sigma_{f, A}(\tau_2)) < d_T(\tau_1, \tau_2)$ for all distinct $\tau_1, \tau_2 \in \T_A$. The same property holds for some iterate $\sigma_{f, A}^{\circ k}$ provided that $f$ is not a $(2, 2, 2, 2)$-map; see Proposition \ref{prop: propeties of pullback map}. In particular, if such a pullback map $\sigma_{f, A}$ has a fixed point, then it is necessarily unique. However, these properties no longer apply for $(2, 2, 2, 2)$-maps. For instance, if $f$ is a $(2, 2, 2, 2)$-map, then $\sigma_{f, P_f}$ is an isometry of $\T_{P_f}$, and it is even possible that $\sigma_{f, P_f} = \id_{\T_{P_f}}$; see \cite[Appendix~C.6]{Hubbard_Book_2}. This creates the main challenge of working with marked Thurston $(2, 2, 2, 2)$-maps, and one of the reasons why Thurston's result does not extend to this context.

Although the pullback maps of transcendental Thurston maps are distance-decreasing, they exhibit more intricate behavior compared with those of finite degree Thurston maps. To explain this further, we need to introduce the \textit{moduli space} $\M_A$, which, roughly speaking, captures all possible complex structures on $S^2 - A$ that are biholomorhic to a punctured Riemann sphere; see Section \ref{subsec: teichmuller and moduli spaces} for the precise definition. There is a natural projection map $\pi_A \colon \T_A \to \M_A$ from the Teichm\"uller space $\T_A$ to the moduli space $\M_A$. For finite degree maps, there is a close interplay between the $\sigma_f$-orbit $(\sigma_{f, A}^{\circ n}(\tau))$ of a point $\tau \in \T_A$ and its projection $(\pi_A(\sigma_{f, A}^{\circ n}(\tau)))$ to the moduli space $\M_A$; see \cite{critically_finite_endomorphisms} and \cite[Section~10.9]{Hubbard_Book_2}. However, this relationship becomes more subtle and less understood in the case when $f$ is a transcendental map (partial results can be found in \cite{small_postsingular_set}). For instance, one of the key steps in the proof of Thurston's characterization theorem is to show that $\sigma_{f, A}$-orbit of $\tau \in \T_A$ converges (indicating that $f$ is realized) if and only if its projection to the moduli space has a limit, provided that $f$ is not a $(2, 2, 2, 2)$-map; see \cite[Theorem 10.9.2]{Hubbard_Book_2}. This result is not yet known for transcendental Thurston maps, posing one of the significant challenges in extending Thurston's original proof to the transcendental setting.

In this paper, we seek to understand the properties of a marked Thurston map, assuming we have knowledge about the same Thurston map with respect to a smaller marked set. Specifically, we discuss  realizability question in this context, employing explicit techniques that shed light on the geometry of the pullback dynamics on Teichmüller and moduli spaces.

\subsection{Main results}

Let $f \colon (S^2, A) \rto$ be a marked Thurston map, with finite or infinite degree, which may have at most countable closed set of essential singularities. Suppose that $B \subset A$ is an $f$-pseudo-invariant subset containing $S_f$. Then, we can also consider the Thurston map $f \colon (S^2, B) \rto$. It is natural to ask how the dynamical properties of the Thurston map $f \colon (S^2, A)\rto$ are related to those of $f \colon (S^2, B) \rto$. For instance, it is clear that if $f \colon (S^2, A) \rto$ is realized, then $f \colon (S^2, B) \rto$ must also be realized, but the converse does not necessarily hold. The main result of this paper provides sufficient and necessary condition for when this happens.

\begin{maintheorem}\label{mainthm A}
    Let $f \colon (S^2, A) \rto$ be a Thurston map of finite or infinite degree. Suppose that $B \subset A$ is an $f$-pseudo-invariant set containing $S_f$. If the Thurston map $f \colon (S^2, B) \rto$ is realized, then the Thurston map $f \colon (S^2, A) \rto$ is realized if and only if it has no degenerate Levy multicurve consisting of curves that are non-essential in $S^2 - B$.
\end{maintheorem}

Here, if $A$ is a finite subset $S^2$, then a simple closed curve $\gamma \subset S^2 - A$ is \textit{non-essential} if it can be shrinked to a point by a homotopy in $S^2 - A$. Otherwise, we say that $\gamma$ is \textit{essential} in $S^2 - A$. A collection $\mathcal{G} = \{\gamma_1, \gamma_2, \dots, \gamma_r\}$ of essential simple closed curves in $S^2 - A$ is called a \textit{Levy cycle} for the Thurston map $f\colon (S^2, A) \rto$ if there exists a cyclic permutation $\sigma$ of $\{1, 2, \dots, r\}$ so that for every $i$ there is a simple closed curve $\widetilde{\gamma}_i \in f^{-1}(\gamma_{i})$ that is homotopic in $S^2 - A$ to $\gamma_{\sigma(i)}$ and $\deg(f|\widetilde{\gamma}_i \colon \widetilde{\gamma}_i \to \gamma_i) = 1$. 
A Levy cycle $\mathcal{G}$ is \textit{degenerate} if each $\gamma_i$ bounds an open Jordan region $D_i$ so that for every~$i$ some connected component $\widetilde{D}_i$ of $f^{-1}(D_i)$ is homotopic in $S^2 - A$ to $D_{\sigma(i)}$ and maps with degree~$1$ by the map $f$ to $D_i$. Finally, Levy cycle $\mathcal{G}$ is called a \textit{Levy multicurve} if all curves in $\mathcal{G}$ are pairwise disjoint and pairwise non-homotopic in $S^2 - A$; see Section \ref{subsec: obstructions} for further discussion.

Main Theorem \ref{mainthm A} treats the following three cases for the map $f$ in a unified manner:
\begin{enumerate}
    \item\label{it: easy case} $f$ is a finite degree map that is not a $(2, 2, 2, 2)$-map;
    \item $f$ is a $(2, 2, 2, 2)$-map;
    \item $f$ is an infinite degree map.
\end{enumerate}

In scenario \eqref{it: easy case}, the Main Theorem \ref{mainthm A} can be derived as a direct consequence of Thurston's characterization criterion for marked Thurston map. However, for the other two cases, Thurston's result does not apply, and we have to address the challenges discussed in Section~\ref{subsec: intro pullback maps}. 

Note that Main Theorem \ref{mainthm A} does not make any assumptions about the function-theoretic properties of the considered Thurston maps. This differs from other results in this direction: \cite{HSS}, which focuses on exponential Thurston maps, \cite{Shemyakov_Thesis}, which is primarily devoted to \textit{structurally finite} Thurston maps, and \cite{small_postsingular_set}, where all Thurston maps have at most three singular values. Note that the geometric and analytic properties of even entire or meromorphic maps with finitely many singular values can be highly non-trivial; see~\cite{qc_foldings, order_conjecture, models_for_class_S}.

\subsubsection{Strategy of proof}\label{subsubsec: Strategy of proof}

For simplicity, we will outline how to derive the existence of a Levy cycle for the Thurston map $f \colon (S^2, A) \rto$ in the context of Main Theorem \ref{mainthm A}. With further efforts, the complete version of Main Theorem \ref{mainthm A} can be established.

We begin by considering the case where the set $C := A - B$ consists of a single fixed point~$c$ of the map $f$. Suppose that the Thurston map $f \colon (S^2, B) \rto$ is realized by a postsingularly finite holomorphic map $g \colon (\widehat{\C}, P) \rto$, and let $\mu \in \T_B$ be the corresponding fixed point of the pullback map $\sigma_{f, B}$ defined on the Teichm\"uller space $\T_B$. There exist a holomorphic submersion $t_{A, B} \colon \T_A \to \T_B$ naturally induced by forgetting the marked point~$c$. It is straightforward to show that the fiber $\Delta := t_{A, B}^{-1}(\mu)$ is an $\sigma_{f, A}$-invariant subset of the Teichm\"uller space $\T_A$. Furthermore, as the following proposition indicates, the pullback map $\sigma_{f, A}$ behaves particularly well on~$\Delta$; see Proposition \ref{prop: extra marked point} for a more detailed version of this result.

\begin{proposition}\label{prop: first prop intro}
    \begin{enumerate}
        \item $\Delta$ is a properly and holomorphically embedded unit disk in $\T_A$ such that $\sigma_{f, A}(\Delta) \subset \Delta$;
        \item there exists a holomorphic covering map $\pi \colon \Delta \to \widehat{\C} - P$ such that diagram \eqref{fig: intro} commutes, where $W := \widehat{\C} - \overline{g^{-1}(P)}$. 

        \begin{figure}[h]
\begin{tikzcd}
    \Delta \arrow[rr, "\sigma_{f, A}|\Delta"] \arrow[dd, "\pi"] & & \sigma_{f, A}(\Delta) \arrow[dd, "\pi"]\\
    & & \\
    \widehat{\C} - P & & W \arrow[ll, "g|W"]
\end{tikzcd}
\caption{Behavior of $\sigma_{f, A}$ on $\Delta$.}\label{fig: intro}
\end{figure}   
    \end{enumerate}
\end{proposition}

In other word, although the pullback map $\sigma_{f, A}$ is defined on the $(|A| - 3)$-complex-dimensional manifold $\T_A$, it admits a one-complex-dimensional invariant subset. This allows us to use powerful tools of one-dimensional complex dynamics and hyperbolic geometry to better understand the dynamical behavior of $\sigma_{f, A}$.

Holomorphic self-maps of the unit disk that satisfy commutative diagram \eqref{fig: intro} have been already investigated. Notably, pullback maps defined on the one-complex-dimensional Teichm\"uller space, which is biholomorphic to the unit disk, frequently exhibit this property; see \cite{critically_finite_endomorphisms} for various examples and extending this observation to higher-dimensional Teichm\"uller spaces in the context of finite degree Thurston maps, and \cite{small_postsingular_set} for the broad family of such pullback maps for infinite degree Thurston maps.  In particular, by applying \cite[Theorem 3.10]{small_postsingular_set}, we conclude that if the Thurston map $f \colon (S^2, A) \rto$ is obstructed, then for any $\tau \in \Delta$, the projection $(\pi(\sigma_{f, A}^{\circ n}(\tau)))$ of the $\sigma_{f, A}$-orbit of $\tau$ converges to the same repelling fixed point $x \in P$ of the map $g$, regardless of the choice of $\tau$. Finally, we develop tools (see Section \ref{subsec: detecting levy cycles}) that allow us to deduce the existence of a Levy cycle for the Thurston map $f \colon (S^2, A) \rto$ based on the behavior of $\sigma_{f, A}|\Delta$.

These ideas also extend to the case when the set $C = B - A$ consists of several fixed points of the map $f$. We say that a point $c \in A$ is a \textit{trivial marked point} of the Thurston map $f \colon (S^2, A) \rto$ if $c \in A - P_f$ and either $c$ is strictly pre-periodic, or it eventually lands to an essential singularity of the map $f$ after several iterations (this includes the case when $c$ is already an essential singularity of $f$). In other words, a point $c \in A - P_f$ is non-trivial if it is periodic under the map $f$. In Section \ref{subsec: forgetting marked points}, we demonstrate that extending the marked set by adding trivial marked points essentially does not change the properties of the Thurston map or the corresponding pullback map, allowing us to ignore them in the context of the proof of Main Theorem \ref{mainthm A}. Therefore, from this point forward, we assume that every point $c \in C$ is periodic under the map $f$.

Given the previous assumptions, we can find an iterate $f^{\circ m} \colon (S^2, A) \rto$ of the original Thurston map such that every point $c \in C$ becomes fixed, which allows us to apply previous considerations to the map $f^{\circ m}$. It is straightforward to see that the existence of a Levy cycle for the Thurston map $f^{\circ m} \colon (S^2, A) \rto$ also implies the existence of one for the original Thurston map $f \colon (S^2, A) \rto$. However, to complete the argument, we must show that $f^{\circ m} \colon (S^2, A) \rto$ is obstructed if $f \colon (S^2, A) \rto$ is also obstructed. This is a trivial result for Thurston maps that are not of $(2, 2, 2, 2)$-type because of the contraction properties of the corresponding pullback maps. However, for the family of marked Thurston $(2, 2, 2, 2)$-maps, this conclusion is non-trivial. Nevertheless, we confirm the positive answer in the following statement, which finishes the proof of Main Theorem \ref{mainthm A}; see Section \ref{subsec: passing to an iterate}, particularly Proposition \ref{prop: passing to iterate}.

\begin{proposition}\label{prop: second prop intro}
    For each $m \geq 1$, a Thurston map $f \colon (S^2, A) \rto$ is realized if and only the iterate $f^{\circ m} \colon (S^2, A) \rto$ is realized. 
\end{proposition}

In conclusion, we mention that the result of Main Theorem \ref{mainthm A} is known for finite degree marked Thurston maps; see \cite[Main Theorem II]{Selinger_Yampolski} and \cite[Theorem E]{Bartholdi_Dudko}. However, our approach differs significantly from these two works, and most crucially, applies to the infinite degree case.

\subsection{Organization of the paper}\label{subsec: Organization of the paper}

The paper is organized as follows. In Section \ref{sec: background}, we review some general background. In Section \ref{subsec: notation and basic concepts}, we fix the notation and state some basic definitions. We discuss topologically holomorphic maps in Section \ref{subsec: topologically holomorphic maps}. The necessary background on Thurston maps is covered in Section \ref{subsec: thurston maps}. Section \ref{subsec: teichmuller and moduli spaces} introduces the Teichm\"uller and moduli spaces of a marked topological sphere. In Section \ref{subsec: pullback maps}, we define pullback maps, discuss their basic properties and their relations with the associated Thurston~maps. Finally, Section \ref{subsec: obstructions} is devoted to Levy cycles.

In Section \ref{sec: marked Thurston maps}, we investigate the realizability question and obstructions for marked Thurston maps. Specifically, in Section \ref{subsec: detecting levy cycles}, we provide tools for detecting Levy cycles for marked Thurston maps.
We then focus on the relations between the properties of Thurston maps $f \colon (S^2, A) \rto$ and $f \colon (S^2, B) \rto$, where $B \subset A$ is an $f$-pseudo-invariant subset containing $S_f$. In Section~\ref{subsec: forgetting marked points}, we consider the case when the set $A - B$ consists of trivial marked points. Section~\ref{subsec: extra fixed point} addresses the situation when $A - B$ consists of a single fixed point of the map $f$. In Section~\ref{subsec: realized thurston maps with extra marked points}, we explore the relations between the fixed points of the pullback maps $\sigma_{f, A}$ and $\sigma_{f, B}$. In Section~\ref{subsec: passing to an iterate}, we use previously established results to relate the properties of a marked Thurston map to those of its iterates. Finally, in Section~\ref{subsec: thurston maps with extra marked points}, we study obstructions for the Thurston map $f \colon (S^2, A) \rto$ to be realized, under the assumption that the Thurston map $f \colon (S^2, B) \rto$ is realized. In particular,  we establish the proof of Main Theorem~\ref{mainthm A}.

\textbf{Acknowledgments.} I would like to express my deep gratitude to my thesis advisor, Dierk Schleicher, for introducing me to the fascinating world of Transcendental Thurston Theory. I am also profoundly thankful to Kevin Pilgrim and Lasse Rempe for their valuable suggestions and for many helpful and inspiring discussions. I would like to thank \textit{Centre National de la Recherche Scientifique} (\textit{CNRS}) for supporting my visits to the University of Saarland and University of Liverpool, where these conversations took place. Special thanks go to Anna Jov\'e and Zachary Smith for the discussions on the dynamics of holomorphic self-maps of the unit disk and their diverse applications. Lastly, I would like to thank Kostiantyn Drach and Dzmitry Dudko for valuable remarks.

\section{Background} \label{sec: background}

\subsection{Notation and basic concepts} \label{subsec: notation and basic concepts}

The sets of positive integers, real and complex numbers are denoted by $\N$, $\R$, and $\C$, respectively. We use the notation $\I:=[0,1]$ for the closed unit interval on the real line, $\D :=\{z \in \C\colon |z|<1\}$ for the
open unit disk in the complex plane, and $\widehat{\C}:=\C\cup\{\infty\}$ for the Riemann sphere. If $m, n \in \N$, the greatest common divisor of $m$ and $n$ is denoted by $\mathrm{gcd}(m, n)$.

We denote the oriented 2-dimensional sphere by $S^2$. In this paper, we treat it as purely topological object. In particular, our convention is to write $g\colon \C \to \widehat{\C}$ or $g \colon \widehat{\C} \to \widehat{\C}$ to indicate that $g$ is holomorphic, and $f\colon S^2 \to S^2$ if $f$ is only continuous. The same rule applies to the notation $g \colon \widehat{\C} \dto \widehat{\C}$ and $f \colon S^2 \dto S^2$ (see Section \ref{subsec: topologically holomorphic maps} for the details).

The cardinality of a set $A$ is denoted by $|A|$ and the identity map on $A$ by $\id_A$. If $f\colon U\to V$ is a map and $W\subset U$, then $f|W$ stands for the restriction of $f$ to $W$. If $U$ is a topological space and $W\subset U$, then $\overline W$ denotes the closure and $\partial W$ the boundary of $W$ in $U$. 

A subset $D$ of $S^2$ is called an \emph{open} or \emph{closed Jordan region} if there exists an injective continuous map $\varphi\colon \overline{\D} \to S^2$ such that $D = \varphi(\D)$ or $D = \varphi(\overline{\D})$, respectively. In this case, $\partial D = \varphi(\partial \D)$ is a simple closed curve in $S^2$. A domain $U \subset \widehat{\C}$ is called an \textit{annulus} if $\widehat{\C} - U$ has two connected components. The \textit{modulus} of an annulus $U$ is denoted by $\mod(U)$ (see \cite[Proposition~3.2.1]{Hubbard_Book_1}~for~the~definition).

Let $U$ and $V$ be topological spaces. A continuous map $H\colon U\times \I \to V$ is called a \emph{homotopy} from $U$ to $V$. When $U = V$, we simply say that $H$ is a homotopy \textit{in} $U$. Given a homotopy $H \colon U \times \I \to V$, for each $t \in \I$, we associate the \emph{time-$t$ map} $H_t:=H(t,\cdot)\colon U\to V$. Sometimes it is convenient to think of the homotopy $H$ as a continuous family of its time maps~$(H_t)_{t \in \I}$. The homotopy $H$ is called an (\emph{ambient}) \emph{isotopy} if the map $H_t\colon U\to V$ is a homeomorphism for each $t\in \I$. Suppose $A$ is a subset of $U$. An isotopy $H\colon U\times \I \to V$ is said to be an isotopy \emph{relative to $A$} (abbreviated ``$H$ is an isotopy rel.\ $A$'') if $H_t(p) = H_0(p)$ for all $p\in A$ and $t\in \I$. 

Given $M, N \subset U$, we say that \emph{$M$ is homotopic} (\textit{in $U$}) to $N$ if there exists a homotopy $H\colon U\times \I \to U$  with $H_0 = \id_U$ and $H_1(M) = N$. Two homeomorphisms $\varphi_0, \varphi_1\colon U\to V$ are called \emph{isotopic} (\emph{rel.\ $A\subset U$}) if there exists an
isotopy $H\colon U\times \I \to V$ (rel.\ $A$) with $H_0=\varphi_0$ and $H_1 = \varphi_1$. 

We assume that every topological surface $X$ is oriented. We denote by $\Homeo^+(X)$ and $\Homeo^+(X, A)$ the group of all orientation-preserving self-homeomorphisms of $X$ and the group of all orientation-preserving self-homeomorphisms of $X$ fixing $A$ pointwise, respectively. We use the notation $\Homeo^+_0(X, A)$ for the subgroup of $\Homeo^+(X, A)$ consisting of all homeomorphisms isotopic rel.\ $A$ to $\id_X$.

Further, we will require the following two topological facts. Here, $\Isol(X)$ denotes the set of isolated points of the topological space $X$.

\begin{lemma}\label{lemm: small sets}
    If $E$ is at most countable closed set of $S^2$, then $E$ coincides with the closure of its isolated points, i.e., $E = \overline{\Isol(E)}$, and its complement $S^2 - E$ is path-connected.
\end{lemma} 

\begin{proof}
    Let $E - \Isol(E) = \{e_1, e_2, \dots, e_n, \dots\}$, and define $E_i := E - \{e_i\}$ for each $i \in \N$. Clearly, each $E_i$ is open and dense in $E$. Therefore, by the Baire category theorem, $\Isol(E) = \bigcap_{i \in \N} E_i$ is dense in $E$.

    The set $S^2 - E$ is path-connected because, for any two points $e_1, e_2 \in S^2 - E$, there are uncountably many arcs in $S^2$ connecting $e_1$ to $e_2$ and intersecting each other only at their endpoints. 
\end{proof}

\begin{lemma}\label{lemm: extensions}
    Let $\varphi \colon X \to Y$ be a homeomorphism, where $X = S^2 - E_X$ and $Y = S^2 - E_Y$, with $E_X$ and $E_Y$ being at most countable closed sets of $S^2$. Then $\varphi$ has a unique continuous extension to $S^2$. Moreover, this extension is a self-homemorphism of $S^2$ with $\varphi(E_X) = E_Y$. 

    If $X = \widehat{\C} - E_X$ and $Y = \widehat{\C} - E_Y$, where $E_X$ and $E_Y$ are at most countable closed sets of~$\widehat{\C}$, and $\varphi \colon X \to Y$ is a biholomorphism, then $\varphi$ extends to a M\"obius transformation.
\end{lemma}

\begin{proof}
    Let $p \in E_X$. We can select a sequence of closed Jordan regions $(U_n)_{n \in \N}$ such that $p \in U_n$, $U_{n + 1} \subset U_n$, and $\partial U_n \subset X$ for every $n \in \N$, and $\bigcap_{n \in \N} U_n = \{p\}$.
    Let $V_n := \overline{\varphi(U_n - X)}$ be a closed Jordan region with $\partial V_n = \varphi(\partial U_n)$. Denote $V := \bigcap_{n \in \N} V_n$.

    \begin{claim}
      The set $V$ is a singleton.  
    \end{claim}

    \begin{subproof}
        Note that $V$ is non-empty due to Cantor's intersection lemma. Next, we show that $V$ is connected. Suppose, for the sake of contradiction, that $V = V' \cup V''$, where $V'$ and $V''$ are both open and closed sets of $V$, and $V' \cap V'' = \emptyset$. Since $V$ is closed in $S^2$, then $V'$ and $V''$ must also be closed subsets of $S^2$. Since $S^2$ is a normal topological space, we can find two open sets $W' \supset V'$ and $W'' \supset V''$ such that $W' \cap W'' = \emptyset$. 
        
        Denote $W := W' \cup W''$. We aim to show that $V_n \subset W$ for sufficiently large $n$. Suppose, on the contrary, that it is not the case. Then the compact set $W_n := V_n - W$ is non-empty for every $n \in \N$. Since, $W_{n + 1} \subset W_n$, Cantor's intersection lemma guarantees the existence a point $q \in \bigcap_{n \in \N} W_n$. However, it contradicts to the fact that $\bigcap_{n \in \N} W_n = V - W = \emptyset$.
        Therefore, $V_n \subset W$ for sufficiently large $n$. But then $V_n$ can be written as a disjoint union of two open non-empty sets $V_n \cap W'$ and $V_n \cap W''$, and we obtain a contradiction with the fact that $V_n$ is connected.

        Suppose there exists a point $q \in V \cap Y$. Then the point $\varphi^{-1}(q) \in X$ belongs to $U_n$ for all $n \in \N$, which is impossible. Hence, $V \subset E_Y$. But $E_Y$ is totally disconnected, and therefore, $V$ must be a singleton.
    \end{subproof}

Now let us define $\varphi(p)$ by setting $\varphi(p) := q$, where $\{q\} = V = \bigcap_{n \in \N}V_n$, and apply the same procedure to every point $p \in E_X$. It is straightforward to verify that this construction is independent of the choice of the sequence $(U_n)$. Indeed, it easily follows from the fact that if $D$ is a neighborhood of $p$, then $U_n \subset D$ for all sufficiently large $n$. At the same time, one can prove that $\varphi$ is continuous on $S^2$ by showing that $\lim_{n \to \infty} f(p_n) = f(p)$ for any sequence $(p_n)$ of points in~$S^2$ converging to $p \in S^2$. Clearly, this is the unique continuous extensions of $\varphi$ to $S^2$.

Suppose there are two points $p_1, p_2 \in S^2$ such that $f(p_1) = f(p_2)$. We can pick two disjoint open Jordan regions $D_1$ and $D_2$ such that $p_i \in D_i$ and $\partial D_i \subset X$ for $i = 1, 2$. In this case, $f(D_1)$ and $f(D_2)$ should intersect at uncountably many points of $S^2$, and that would imply that the original map was not injective on $X$. Thus, $\varphi$ is injective on $S^2$. Now, let $q \in S^2$ be an arbitrary point and $(q_n)$ be a sequence of points in $Y$ converging to $q$. Then the sequence $(\varphi^{-1}(q_n))$ has a converging subsequence in~$S^2$. Without loss of generality, we assume that the entire sequence $(\varphi^{-1}(q_n))$ converges to $p \in S^2$. Since $\varphi \colon S^2 \to S^2$ is continuous, it follows that $\varphi(p) = q$. Thus, $\varphi \colon S^2 \to S^2$ is surjective and, because $S^2$ is both compact and Hausdorff, $\varphi$ is a self-homeomorphism of $S^2$. Obviously, $\varphi(E_X) = E_Y$.

Now, suppose that $\varphi \colon X \to Y$ is a biholomorphism. As discussed earlier, it has a unique extension to a self-homeomorphism of $\widehat{\C}$. By post-composing $\varphi$ with a M\"obius transformation, we can assume that it restricts to a self-homeomorphism of $\C$ that is holomorphic outside of at most countable closed set. Then it can be shown, using Morera's theorem \cite[Theorem 5.10]{conway}, that $\varphi$ must be holomorphic everywhere on $\widehat{\C}$. We leave this verification to the reader.
\end{proof}

\subsection{Topologically holomorphic maps} \label{subsec: topologically holomorphic maps}

In this section we recall the definition of a topologically holomorphic map and some of its basic properties (for more detailed discussion see \cite[Section 2.3]{our_approx} and \cite[Section 2.2]{small_postsingular_set}; see also \cite{Stoilow}).

\begin{definition}\label{def: top hol}
    Let $X$ and $Y$ be two connected topological surfaces. A map $ f\colon X \to Y$ is called \textit{topologically holomorphic} if it satisfies one of the following four equivalent conditions:
    \begin{enumerate}
        \item for every $p \in X$ there exist $d \in \N$, a neighborhood $U$ of $x$, and two orientation-preserving homeomorphisms $\psi\colon U \to \D$ and $\varphi\colon f(U) \to \D$ such that $\psi(p) = \varphi(f(p)) = 0$ and $(\varphi \circ f \circ \psi^{-1})(z) = z^d$ for all $z \in \D$;

        \item $f$ is continuous, open, \textit{discrete} (i.e., $f^{-1}(q)$ is discrete in $X$ for very $q \in Y$), and for every $p \in X$ such that $f$ is locally injective at $p$, there exists a neighborhood $U$ of $p$ for which $f|U\colon U \to f(U)$ is an orientation-preserving homeomorphism;

        \item there exist Riemann surfaces $S_X$ and $S_Y$ and orientation-preserving homeomorphisms $\varphi \colon Y \to S_Y$ and $\psi \colon X \to S_X$ such that $\varphi \circ f \circ \psi^{-1} \colon S_X \to S_Y$ is a holomorphic map;

        \item \label{it: pulling back} for every orientation-preserving homeomorphism $\varphi \colon Y \to S_Y$, where $S_Y$ is a Riemann surface, there exist a Riemann surface $S_X$ and an orientation-preserving homeomorphism $\psi \colon X \to S_X$ such that $\varphi \circ f \circ \psi^{-1} \colon S_X \to S_Y$ is a holomorphic map.
    \end{enumerate}
\end{definition}

\begin{remark}\label{rem: uniqueness}
    Note that in condition (\ref{it: pulling back}) of Definition \ref{def: top hol}, the pair $(\psi, S_X)$ is uniquely determined up to the following equivalence: $(\psi, S_X) \sim (\widetilde{\psi}, \widetilde{S}_X)$ if there exists a biholomorphism $\theta \colon S_X \to \widetilde{S}_X$ such that $\psi = \theta \circ \widetilde{\psi}$. In particular, the conformal type of the surface $S_X$ is uniquely determined by the homeomorphism $\varphi \colon Y \to S_Y$. If $\varphi$ and the surface~$S_X$ are fixed, then the homeomorphism $\psi \colon X \to S_X$ is uniquely determined up to post-composition with a biholomorphism of $S_X$.
\end{remark}

It is straightforward to define the concepts of \textit{regular}, \textit{singular}, \textit{critical}, and \textit{asymptotic values}, as well as \textit{regular} and \textit{critical points} and their \textit{local degrees} (denoted by $\deg(f, \cdot)$) for the topologically holomorphic map $f$ (see \cite[Definition 2.7]{our_approx}). We denote by $S_f \subset Y$ the \textit{singular set} of~$f$, i.e., the union of all singular values of the topologically holomorphic map $f \colon X \to Y$. We say that the map $f$ is \textit{of finite type} or belongs to the \textit{Speiser class} $\mathcal{S}$ if the set $S_f$ is finite.

It is possible to derive the following isotopy lifting property for finite-type topologically holomorphic maps (cf. \cite[Propostion 2.3]{Rempe_Epstein} and \cite[Proposition 2.2]{small_postsingular_set}).

\begin{proposition}\label{prop: isotopy lifting property}
    Let $f \colon X \to Y$ and $\widetilde{f}\colon \widetilde{X} \to \widetilde{Y}$ be topologically holomorphic maps of finite type, where $X$, $Y$, $\widetilde{X}$, and $\widetilde{Y}$ are connected topological surfaces. Suppose that $\varphi_0 \circ f = \widetilde{f} \circ \psi_0$ for some orientation-preserving homeomorphisms $\varphi_0 \colon Y \to \widetilde{Y}$ and $\psi_0 \colon X \to \widetilde{X}$.
    Let $A \subset Y$ be a finite set containing $S_{f}$ and $\varphi_1 \colon Y \to \widetilde{Y}$ be an orientation-preserving homeomorphism isotopic rel.\ $A$ to $\varphi_0$. Then $\varphi_1 \circ f = \widetilde{f} \circ \psi_1$ for some orientation-preserving homeomorphism $\psi_1 \colon X \to \widetilde{X}$ isotopic rel.\ $f^{-1}(A)$ to $\psi_0$.
\end{proposition}

\begin{proof}
    Let $(\varphi_t)_{t \in \I}$ be the corresponding isotopy. From the definition of a singular value, it follows that the restrictions
    $\varphi_t \circ f | Z  \colon Z \to T$ are covering maps for each $t \in \I$, where $Z := X - f^{-1}(A)$ and $T := Y - A$. Therefore, \cite[Lemma 2.7]{Astorg_Benini_Fagella} implies the existence of an isotopy $(\phi_t)_{t \in \I}$ in~$Z$ such that $\varphi_t \circ f = \varphi_0 \circ f \circ \phi_t$. Each homeomorphism $\phi_t \colon Z \to Z$ extends to a self-homeomorphism of the entire surface $X$ since all points of the set $X - Z = f^{-1}(A)$ are isolated. Moreover, the homotopy $(\phi_t)_{t \in \I}$ can be viewed as an isotopy in $X$ rel.\ ~$f^{-1}(A)$.
    
    At the same time, $\varphi_0 \circ f = \varphi_0 \circ f \circ \phi_0$ and, therefore, we have the following:
    $$
        \varphi_1 \circ f = \varphi_0 \circ f \circ \phi_1 = (\varphi_0 \circ f \circ \phi_0) \circ \phi_0^{-1} \circ \phi_1 = \varphi_0 \circ f \circ \phi_0^{-1} \circ \phi_1 = \widetilde{f} \circ (\psi_0 \circ \phi_0^{-1} \circ \phi_1).
    $$
    Thus, we can set $\psi_1 := \psi_0 \circ \phi_0^{-1} \circ \phi_1$, and $(\psi_0 \circ \phi_0^{-1} \circ \phi_t)_{t \in \I}$ provides the required isotopy rel.\ $f^{-1}(A)$. Finally, $\psi_1$ is orientation-preserving since $f$ and $\widetilde{f}$ are local orientation-preserving homeomorphisms outside the sets of their critical points.
\end{proof}

We say that a topological surface $X$ is a \textit{countably-punctured sphere}, if $X = S^2 - E$, where $E$ is a closed subset of $S^2$ that is at most countable. It is worth noting that we allow the set $E$ to be finite or even empty. In this paper, we mainly focus on topologically holomorphic maps $f \colon X \to S^2$ defined on a countably-punctured sphere $X$. Additionally, we always assume that $f$ cannot be extended as a topologically holomorphic map to a neighborhood of any of the points of the set $S^2 - X$.
For simplicity, we will use the notation $f \colon S^2 \dto S^2$ in order to indicate that $f$ might not be defined on at most countable closed set of $S^2$. Similar to the holomorphic case, any point $e \in S^2 - X$ is referred as an \textit{essential singularity} of the map $f$, which is \textit{isolated} if it is an isolated point of the set $S^2 - X$. We say that a topologically holomorphic map $f \colon S^2 \dto S^2$ is \textit{transcendental} if it has an essential singularity. Given our previous assumptions on the map $f$, this is equivalent to saying that $f$ has infinite topological degree. Likewise, for a holomorphic map $g$ defined everywhere on~$\widehat{\C}$, potentially excluding at most countable closed set of singularities, we use the notation $g \colon \widehat{\C} \dto \widehat{\C}$, and we say that $g \colon \widehat{\C} \dto \widehat{\C}$ is holomorphic.

Let $\varphi\colon S^2 \to \widehat{\C}$ be an orientation-preserving-homeomorphism and $f \colon X \to S^2$ be a topologically holomorphic map of finite type, where $X$ is a countably-punctured sphere. According to item \eqref{it: pulling back} of Definition \ref{def: top hol}, there exists an orientation-preserving homeomorphism $\psi \colon X \to R$, where $R$ is a connected Riemann surface, so that the map $\varphi \circ f \circ \psi^{-1} \colon R \to \widehat{\C}$ is holomorphic. Moreover, according to Remark \ref{rem: uniqueness}, the conformal type of $R$ is uniquely determined for a given homeomorphism $\varphi$. 
\begin{definition}\label{def: admissible type}
   We define $f$ as a topologically holomorphic map of \textit{admissible type} if the Riemann surface $R$, as described above, is biholomorphic to a \textit{countably-punctured Riemann sphere} for any orientation-preserving homeomorphism $\varphi \colon S^2 \to \widehat{\C}$.
\end{definition}

In fact, the following proposition shows that Definition \ref{def: admissible type} is independent of the choice of~$\varphi$.

\begin{proposition}\label{prop: admissible type}
    Suppose that $f \colon X \to S^2$ is a finite-type topologically holomorphic map, where $X$ is a countably-punctured sphere. If $f$ is of admissible type, then for every orientation-preserving homeomorphism $\varphi \colon S^2 \to \widehat{\C}$, there exists an orientation-preserving homeomorphism $\psi \colon X \to R$ such that $R$ is a countably-punctured Riemann sphere and the map $\varphi \circ f \circ \psi^{-1} \colon R \to \widehat{\C}$ is holomorphic. Moreover, such $\psi$ is unique up to post-composition with a M\"obius transformation. 
\end{proposition}

\begin{proof}
    Since $f$ has an admissible type, there exist two orientation-preserving homeomorphisms $\varphi_1 \colon S^2 \to \widehat{\C}$ and $\psi_1 \colon X \to \widehat{\C} - E$, where $E$ is at most countable closed subset of $\widehat{\C}$, such that the map $g_1 := \varphi_1 \circ f \circ \psi_1^{-1} \colon \widehat{\C} - E \to \widehat{\C}$ is holomorphic. Now let $\varphi_2 \colon S^2 \to \widehat{\C}$ be any other orientation-preserving homeomorphism. By item \eqref{it: pulling back} of Proposition \ref{def: top hol}, there exists a connected Riemann surface $R$ and an orientation-preserving homeomorphism $\psi_2 \colon X \to R$ so that the map $g_2 := \varphi_2 \circ f \circ \psi_2^{-1} \colon R \to \widehat{\C}$ is holomorphic.
    In particular, 
    $g_2 = \widetilde{\varphi} \circ g_1 \circ \widetilde{\psi}^{-1}$ on $R$, where $\widetilde{\varphi} := \varphi_2 \circ \varphi_1^{-1} \colon \widehat{\C} \to \widehat{\C}$ and $\widetilde{\psi} := \psi_2 \circ \psi_1^{-1} \colon \widehat{\C} - E \to R$ are orientation-preserving homeomorphisms. Note that $\widetilde{\varphi}$ is isotopic rel.\ $S_f$ to some \textit{diffeomorphism} of $\widehat{\C}$ (see, for instance, \cite[Theorem 1.13]{farb_margalit}). According to Proposition \ref{prop: isotopy lifting property}, without loss of generality, we can assume that $\widetilde{\varphi}$ is itself a diffeomorphism of $\widehat{\C}$, and therefore, it is a \textit{quasiconformal mapping}. Let $\mu_{\widetilde{\varphi}} := \widetilde{\varphi}^*\mu_0$ be the Beltrami form corresponding to $\widetilde{\varphi}$, where $\mu_0$ is the zero Beltrami form on $\widehat{\C}$ (see \cite[Sections 1.2 and 1.3]{branner_fagella_2014}). We define another Beltrami form on $\widehat{\C}$ by
    $$
    \mu(z) = 
    \begin{cases}
        (g_1)^*\mu_{\widetilde{\varphi}}(z), \text{ if } z \in \widehat{\C} - E,\\
        0, \text{ otherwise}.
    \end{cases}
    $$
    By a standard argument (see, for example, \cite[Definition 1.34]{branner_fagella_2014}) involving Riemann Measurable Mapping Theorem \cite[Theorem 1.28]{branner_fagella_2014}, it is easy to establish the existence of a quasiconformal mapping $\theta \in \Homeo^+(\C)$ such that $\mu_{\theta} = \mu$ almost everywhere on $\widehat{\C}$ and the map $g := \widetilde{\varphi} \circ g_1 \circ \theta^{-1} \colon \widehat{\C} - \theta(E) \to \widehat{\C}$ is holomorphic. Thus, applying Remark \ref{rem: uniqueness} to the map $g_1$ and pairs $(\widetilde{\psi}, R)$ and $(\theta, \widehat{\C} - \theta(E))$, we conclude that $R$ must be biholomorphic to a countably-punctured Riemann sphere.

    The uniqueness of $\psi$ as stated in this proposition can be deduced from Lemma~\ref{lemm: extensions} and Remark~\ref{rem: uniqueness}.
\end{proof}

Further, we will focus on finite-type topologically holomorphic maps $f \colon S^2 \dto S^2$ of admissible type. These maps serve as topological models for the family of holomorphic maps $g \colon \widehat{\C} \dto \widehat{\C}$ having finitely many singular values. The most straightforward examples of such maps are branched self-covers of $S^2$. The singular sets of these maps consist of finitely many critical values, and they do not have any essential singularities. A finite-type topologically holomorphic map $f \colon S^2 \dto S^2$ with a single essential singularity is of admissible type if and only if $f$ is of \textit{parabolic type} (see \cite[Section 2.2]{small_postsingular_set}).

Let $f \colon S^2 \dto S^2$ be a finite-type topologically holomorphic map of admissible type. Then Great Picard's Theorem and Lemma \ref{lemm: small sets} imply that in any neighborhood of an essential singularity of such a map $f$, every value is attained infinitely often with at most two exceptions. In particular, $f$ can have at most two \textit{omitted values}, i.e., points~$p$ in $S^2$ such that the preimage $f^{-1}(p)$ is empty. Furthermore, each omitted value is an asymptotic value of~$f$. Additionally, observe that if $A \subset S^2$ is a finite set with $|A| \geq 3$ and $S_f \subset A$, then the~restriction
$$
    f|S^2 - \overline{f^{-1}(A)}\colon S^2 - \overline{f^{-1}(A)} \to S^2 - A
$$
is a covering map. Note that the closure $\overline{f^{-1}(A)}$ consists of $f^{-1}(A)$ along with all essential singularities of the map $f$.

The main motivation for introducing maps with countably many essential singularities is that the family of topologically holomorphic maps with at most one (or even finitely many) essential singularities is not closed under composition. Indeed, it is straightforward to construct two such maps, each with a single essential singularity, whose composition has infinitely many of them. However, the class of finite-type topologically holomorphic maps $f \colon S^2 \dto S^2$ of admissible type is closed under the operation of composition.

\begin{proposition}\label{prop: composition}
    Let $f_1 \colon X_1 \to S^2$ and $f_2 \colon X_2 \to S^2$ be two finite-type topologically holomorphic maps, where $X_1$ and $X_2$ are countably-punctured spheres. Then the composition $f_1 \circ f_2 \colon X_2 \cap f_2^{-1}(X_1) \to S^2$ is also a finite-type topologically holomorphic map, with $S_{f_1 \circ f_2} = S_{f_1} \cup f_1(S_{f_2})$. Moreover, if both $f_1$ and $f_2$ are of admissible type, then the composition $f_1 \circ f_2 \colon S^2 \dto S^2$ is of admissible type as well.
\end{proposition}

\begin{proof}
    The first part of this proposition easily follows from Definition \ref{def: top hol} and the definition of a singular value.

    Let $\varphi \colon S^2 \to \widehat{\C}$ be an orientation-preserving homeomorphism. Then by Lemma \ref{lemm: extensions} and Proposition \ref{prop: admissible type}, there exists another orientation-preserving homeomorphism $\psi \colon S^2 \to \widehat{\C}$ such that $\varphi \circ f \circ \psi^{-1} \colon \psi(X_1) \to \widehat{\C}$ is holomorphic.
    Applying the same argument to the map $f_2$ and the homeomorphism $\psi$, we derive the existence of an orientation-preserving homeomorphism $\theta \colon S^2 \to \widehat{\C}$ such that $\psi \circ f \circ \theta^{-1} \colon \theta(X_2) \to~\widehat{\C}$ is holomorphic. Hence,
    $$
        \varphi \circ (f_1 \circ f_2) \circ \theta^{-1} = (\varphi \circ f_1 \circ \psi^{-1}) \circ (\psi \circ f_2 \circ \theta^{-1}) \colon \theta(X_2 \cap f_2^{-1}(X_1)) \to \widehat{\C}
    $$
    is holomorphic as well. Obviously, $\theta(X_2 \cap f_2^{-1}(X_1))$ is a countably-punctured Riemann sphere, and therefore, $f_1 \circ f_2$ is of admissible type.
\end{proof}

Finally, observe that Lemma \ref{lemm: extensions} allows us to state the following version of Proposition~\ref{prop: isotopy lifting property}.

\begin{corollary}\label{corr: isotopy lifting property}
    Let $f \colon X \to S^2$ and $\widetilde{f}\colon \widetilde{X} \to S^2$ be finite-type topologically holomorphic maps, where $X$ and $\widetilde{X}$ are countably-punctured spheres. Suppose that $\varphi_0 \circ f = \widetilde{f} \circ \psi_0$ for some $\varphi_0, \psi_0 \in \Homeo^+(S^2)$.
    Let $A \subset S^2$ be a finite set containing $S_{f}$ and $\varphi_1 \in \Homeo^+(S^2)$ is isotopic rel.\ $A$ to $\varphi_0$. Then $\varphi_1 \circ f = \widetilde{f} \circ \psi_1$ for some $\psi_1 \in \Homeo^+(S^2)$ isotopic rel.\ $f^{-1}(A) \cup (S^2 - X) \supset \overline{f^{-1}(A)}$ to $\psi_0$.
\end{corollary}

Analogously to \cite[Corollary 2.3]{small_postsingular_set}, we have the following result.

\begin{corollary}\label{corr: homotopy lifting for curves}
    Let $f \colon X \to S^2$ be a finite-type topologically holomorphic map, where $X$ is a countably-punctured topological sphere. Let $A \subset S^2$ be a finite set containing $S_f$, and suppose that $\gamma_0$ is a simple closed curve in $S^2 - A$. Let $\widetilde{\gamma}_0 \subset f^{-1}(\gamma)$ be a simple closed curve with $\deg(f|\widetilde{\gamma}_0\colon \widetilde{\gamma}_0 \to \gamma_0) = d$. If $\gamma_1$ is a simple closed curve that is homotopic in $S^2 - A$ to $\gamma_0$, then there exists a simple closed curve $\widetilde{\gamma}_1 \subset f^{-1}(\gamma_1)$ such that $\widetilde{\gamma}_0$ and $\widetilde{\gamma}_1$ are homotopic in $X -  f^{-1}(A) \subset S^2 - \overline{f^{-1}(A)}$ and $\deg(f|\widetilde{\gamma}_1\colon \widetilde{\gamma}_1 \to \gamma_1) = d$.

    Moreover, suppose that $D_0$ and $\widetilde{D}_0$ are connected components of $S^2 - \gamma_0$ and $S^2 - \widetilde{\gamma}_0$, respectively, such that $f|\widetilde{D}_0 \colon \widetilde{D}_0 \to D_0$ is a homeomorphism. Then there exist connected components $D_1$ and $\widetilde{D}_1$ of $S^2 - \gamma_1$ and $S^2 - \widetilde{\gamma}_1$, respectively, such that $D_0$ is homotopic in $S^2 - A$ to $D_1$, $\widetilde{D}_0$ is homotopic in $X -  f^{-1}(A) \subset S^2 - \overline{f^{-1}(A)}$ to $\widetilde{D}_1$, and $f|\widetilde{D}_1 \colon \widetilde{D}_1 \to D_1$ is a homeomorphism.
\end{corollary}

\subsection{Thurston maps} \label{subsec: thurston maps}

Let $f \colon X \to S^2$ be a topologically holomorphic map, where $X = S^2 - E$ is a countably-punctured sphere. Then we define the orbit $\mathcal{O}_f(p)$ of $p \in S^2$ under~$f$~as
$$
    \mathcal{O}_f(p) := \{q \in S^2: q = f^{\circ n}(p) \text{ for some } n \geq 0\}.
$$
It is worth to note that $f^{\circ n}(p)$ might be defined only for finitely many $n$ if $f^{\circ m}(p) \in E$ for some $m \geq 0$. In this case, we say that the orbit of $p$ \textit{terminates} after $m$ iterations of $f$.

The \textit{postsingular set} $P_f$ of the map $f$ is defined as the union of all forward orbits of the singular values of $f$. We say that $f \colon S^2 \dto S^2$ is \textit{postsingularly finite} (\textit{psf} in short) if the set $P_f$ is finite, i.e., $f$ has finitely many singular values and each of them eventually becomes periodic or lands on an essential singularity of~$f$ under the iteration. 
Postsingularly finite topologically holomorphic maps of finite degree are also called \textit{postcritically finite} (\textit{pcf} in short), and their \textit{postsingular values} are called \textit{postcritical}, as their singular values are always~critical.

\begin{definition}\label{def: pseudo-invariant set}
    For a topologically holomorphic map $f\colon S^2 \dto S^2$, the subset $A \in S^2$ is called \textit{$f$-pseudo-invariant} if the orbit of each $a \in A$ belongs to the set $A$.
\end{definition}


Every $f$-invariant set is $f$-pseudo-invariant and, if the map $f$ has no essential singularities, the converse holds as well. Note that if $|A| \geq 3$ and $f$ is a finite-type topologically holomorphic map of admissible type, then the set $A$ is $f$-pseudo-invariant if and only if $A \subset \overline{f^{-1}(A)}$. The postsingular set $P_f$ provides an example of $f$-pseudo-invariant set.

Now we are ready to state one of the key definitions of this section.

\begin{definition}\label{def: thurston map}
    A non-injective topologically holomorphic map $f\colon S^2 \dto S^2$ is called a \textit{Thurston map} if it is postsingularly finite and of admissible type.
    
    Given a finite $f$-pseudo-invariant set $A \subset S^2$ such that $P_f \subset A$, we call the pair $(f, A)$ a \textit{marked Thurston map} and $A$ its~\textit{marked~set}.
\end{definition}

We often consider marked Thurston maps in the same way as usual Thurston maps and use the notation $f\colon (S^2, A) \righttoleftarrow$ while still assuming that $f$ might not be defined on the entire sphere $S^2$. If no specific marked set is mentioned, we assume it to be $P_f$. 

\begin{definition}\label{def: isotopy of thurston maps}
    Two Thurston maps $f_1\colon (S^2, A) \righttoleftarrow$ and $f_2 \colon (S^2, A) \righttoleftarrow$ are called \textit{isotopic} (\textit{rel.\ $A$}) if there exists $\phi \in \Homeo_0^+(S^2, A)$ such that $f_1 = f_2 \circ \phi$.
\end{definition}

\begin{remark}\label{rem: isotopy of thurston maps}
    Let $f_1\colon (S^2, A) \righttoleftarrow$ and $f_2 \colon (S^2, A) \righttoleftarrow$ be two Thurston maps satisfying the relation $\phi_1 \circ f_1 = f_2 \circ \phi_2$ for some $\phi_1, \phi_2 \in \Homeo_0^+(S^2, A)$. Then it follows from Corollary~\ref{corr: isotopy lifting property} that $f_1$ and $f_2$ are isotopic rel.\ $A$.
\end{remark}

The notion of isotopy for Thurston maps depends on their common marked set. Consequently, we sometimes refer to isotopy \textit{relative $A$} (or rel.\ $A$ for short) to specify which marked set is being considered. This applies to other notions introduced below that also depend on the choice of the marked set.

We say that two (marked) Thurston maps are \textit{combinatorially equivalent} if they are ``topologically conjugate up to isotopy'':

\begin{definition}\label{def: comb equiv}
    Two Thurston maps $f_1 \colon (S^2, A_1) \rto$ and $f_2\colon (S^2, A_2) \rto$ are called combinatorially (or \textit{Thurston}) equivalent if there exist two Thurston maps $\widetilde{f}_1 \colon (S^2, A_1) \rto$ and $\widetilde{f}_2\colon (S^2, A_2) \rto$ such that:
    \begin{itemize}
        \item $f_i$ and $\widetilde{f}_i$ are isotopic rel.\ $A_i$ for each $i = 1, 2$, and

        \item $\widetilde{f}_1$ and $\widetilde{f}_2$ are conjugate via a homeomorphsim $\phi \in \Homeo^+(S^2)$, i.e., $\phi \circ \widetilde{f}_1 = \widetilde{f}_2 \circ \phi$, such that $\phi(A_1) = A_2$.
    \end{itemize}
\end{definition}

\begin{remark}\label{rem: comb equiv}
    Definition \ref{def: comb equiv} can be reformulated in a more classical way. Thurston maps $f_1 \colon (S^2, A_1) \rto$ and $f_2\colon (S^2, A_2) \rto$ are combinatorially equivalent if and only if there exist two homeomorphisms $\phi_1, \phi_2 \in \Homeo^+(S^2)$ such that $\phi_1(A_1) = \phi_2(A_1) = A_2$, $\phi_1$ and $\phi_2$ are isotopic rel.\ $A$, and $\phi_1 \circ f_1 = f_2 \circ \phi_2$.
\end{remark}

\begin{remark}
    If $A_1 = P_{f_1}$ and $A_2 = P_{f_2}$, then the condition that $\phi(A_1) = A_2$ in Definition~\ref{def: comb equiv} and the condition $\phi_1(A_1) = \phi_2(A_1) = A_2$ in Remark \ref{rem: comb equiv} can be removed since they are automatically satisfied if all other conditions hold.
\end{remark}

A Thurston map $f\colon (S^2, A)\righttoleftarrow$ is said to be \textit{realized} if it is combinatorially equivalent to a postsingularly finite holomorphic map $g \colon (\widehat{\C}, P) \rto$. If $f \colon (S^2, A) \rto$ is not realized, we say that it is \textit{obstructed}.

In the following statement we explore the case when the marked set of a Thurston map consists of at most three points.

\begin{proposition}\label{prop: small set}
    Let $f \colon (S^2, A) \rto$ be a Thurston map. The marked set $A$ must contain at least two points. Furthermore:
    \begin{enumerate}
        \item \label{it: three points} if $|A| \leq 3$, then $f \colon (S^2, A) \rto$ is realized, and
        \item \label{it: two points} if $|A| = 2$, then $f \colon (S^2, A) \rto$ is combinatorially equivalent to either $z \mapsto z^d\colon (\widehat{\C}, \{0, \infty\}) \rto$ or $z \mapsto z^{-d}\colon (\widehat{\C}, \{0, \infty\}) \rto$, where $d := \deg(f)$.
    \end{enumerate}
\end{proposition}

\begin{proof}
    Let $X$ be the complement of the set of essential singularities of $f$. Then the restriction $f | X - f^{-1}(A) \colon X - f^{-1}(A) \to S^2 - A$ is a covering map. If $A = \emptyset$ or $|A| = 1$, this covering map must be a homeomorphism, since $S^2 - A$ is simply connected, and $X - f^{-1}(A)$ would either be $S^2$ or a once-punctured sphere, respectively. In either case, it is clear that $f$ must be a homeomorphism of $S^2$, which is impossible for a Thurston map. Thus, the set $A$ contains at least two points.

    If $|A| = 2$, classical theory of covering map gives two possibilities: either $X - f^{-1}(A)$ is a once-punctured sphere and the covering $f | X - f^{-1}(A) \colon X - f^{-1}(A) \to S^2 - A$ has infinite degree, or $X - f^{-1}(A)$ is a twice-punctured sphere and the corresponding covering has a finite degree. In the first case, $f$ has an essential singularity, implying that $S^2 - X$ consists of a single point and $f^{-1}(A)$ is empty. However, this leads to a contradiction, as there exists a point $a \in A$ that is not an essential singularity of $f$ and, in particular, $f(a) \in A$. 
    
    In the second case, $f$ has a finite degree and two critical values $a, b \in A = S_f = P_f$ with $f^{-1}(\{a, b\}) = \{a, b\}$. By item~\eqref{it: pulling back} of Definition \ref{def: top hol}, there exist two orientation-preserving homeomorphisms $\varphi \colon S^2 \to \widehat{\C}$ and $\psi \colon S^2 \to \widehat{\C}$ such that $f = \varphi \circ g \circ \psi^{-1}$, where $g$ is a rational map. Clearly, $|S_g| = 2$, and therefore, we can post-compose $\varphi$ and $\psi$ with M\"obius transformations, to ensure that $S_g = \{0, \infty\}$, $\varphi(a) = \psi(a) = 0$, and $\varphi(b) = \psi(b) = \infty$. It follows that $g(z) = z^{\pm d}$, where $d = \deg(f)$. According to \cite[Proposition 2.3]{farb_margalit}, $\varphi$ and $\psi$ are isotopic rel.\ $A$ and the rest of item \eqref{it: two points} follows.
    
    Similarly, by \cite[Proposition 2.3]{farb_margalit}, two orientation-preserving homeomorphisms $\varphi \colon S^2 \to \widehat{\C}$ and $\psi \colon S^2 \to \widehat{\C}$ that agree on the set $A \subset S^2$, where $|A| \leq 3$, are isotopic rel.\ $A$. Just as in the case when $|A| = 2$, this observation implies that any Thurston map $f\colon (S^2, A) \rto$ is realized when the marked set $A$ contains three or fewer points, leading to item \eqref{it: three points}.
\end{proof}

\subsection{Teichm\"uller and moduli spaces} \label{subsec: teichmuller and moduli spaces}

Let $A \subset S^2$ be a finite set containing at least three points. Then the \textit{Teichm\"{u}ller space of the sphere $S^2$ with the marked set $A$} is defined as
$$
    \T_A := \{\varphi\colon S^2 \rightarrow \widehat{\C} \text{ is an orientation-preserving homeomorphism}\} / \sim
$$
where $\varphi_1 \sim \varphi_2$ if there exists a M\"obius transformation $M$ such that $\varphi_1$ is isotopic rel.\ $A$ to~$M \circ \varphi_2$. 

Similarly, we define the \textit{moduli space of the sphere $S^2$ with the marked set $A$}:
$$
    \M_A := \{\eta\colon A \rightarrow \widehat{\C} \text{ is injective}\} / \sim,
$$
where $\eta_1 \sim \eta_2$ if there exists a M\"obius transformation $M$ such that $\eta_1 = M \circ \eta_2$. 

Further, $[\cdot]_{A}$ denotes an equivalence class corresponding to a point of either the Teichm\"uller space $\T_A$ or the moduli space $\M_A$. In situations when the considered marked set is obvious, we simply use the notation $[\cdot]$.
Note that there is a natural map $\pi_A \colon \T_A \to \M_A$ defined as $\pi_A([\varphi]) = [\varphi|A]$. According to \cite[Proposition 2.3]{farb_margalit}, when $|A| = 3$, both the Teichm\"uller space $\T_A$ and the moduli space $\M_A$ are just single points. Therefore, for the rest of this section, we assume that $|A| \geq 4$. 

It is known that the Teichm\"uller space $\T_A$ admits a complete metric $d_T$, known as the \textit{Teichm\"uller metric} \cite[Proposition~6.4.4]{Hubbard_Book_1}. Moreover, with respect to the topology induced by this metric, $\T_A$ is a contractible space \cite[Corollary 6.7.2]{Hubbard_Book_1}.  At the same time, both $\T_A$ and $\M_A$ admit structures of $(|A| - 3)$-complex manifolds (see \cite[Theorem 6.5.1]{Hubbard_Book_1}) so that the map $\pi_A \colon \T_A \to \M_A$ becomes a holomorphic universal covering map \cite[Section 10.9]{Hubbard_Book_2}. 

Moreover, the complex structure of $\M_A$ is quite explicit in the general case. Let $A = \{a_1, a_2, \dots, a_k, a_{k + 1}, a_{k + 2}, a_{k + 3}\}$, $k \geq 1$, where the indexing of the points of $A$ is chosen arbitrarily. Define the map $h \colon \M_A \to \C^{k} - \L_k$ by
$$
    h([\varphi]) = (\varphi(a_1), \varphi(a_2), \dots, \varphi(a_k)),
$$
where the representative $\varphi \colon S^2 \to \widehat{\C}$ is chosen so that $\varphi(a_{k + 1}) = 0, \varphi(a_{k + 2}) = 1$, and $\varphi(a_{k + 3}) = \infty$, and where $\L_k$ is the subset of $\C^k$ defined by
$$
    \L_k := \{(z_1, z_2, \dots, z_k) \in \C^k: z_i = z_j \text{ for some $i \neq j$}, \text{ or } z_i = 0, \text{ or } z_i = 1\}.
$$
It is known that the map $h$ provides a biholomorphism between $\M_A$ and $\C^k - \L_k$ (see \cite[Section 10.9]{Hubbard_Book_2}).

The case $|A| = 4$ is particularly simple. In this situation, the Teichm\"uller space~$\T_A$ is biholomorphic to~$\D$, with the metric $d_T$ coinciding with the usual hyperbolic metric on the unit disk (\cite[Corollary 6.10.3 and Theorem 6.10.6]{Hubbard_Book_1}). Additionally, the moduli space~$\M_A$ is biholomorphic to the three-punctured Riemann sphere $\widehat{\C} - \{0, 1, \infty\}$.

Suppose that $A$ and $B$ are finite subsets of $S^2$ such that $B \subset A$. There exist two maps $t_{A, B} \colon \T_A \to \T_B$ and $m_{A, B} \colon \M_A \to \M_B$ that are naturally induced by forgetting the marked points of the set $A - B$. More precisely, $t_{A, B}([\varphi]_{A}) = [\varphi]_{B}$ and $m_{A, B}([\eta]_{A}) = [\eta]_{B}$, where $[\varphi]_A \in \T_A$, $[\varphi]_B \in \T_B$, $[\eta]_A \in \M_A$, and $[\eta]_B \in \M_B$.

Suppose that $A = \{a_1, a_2, \dots, a_k, a_{k + 1}, a_{k + 2}, a_{k + 3}\}$ and $B = \{a_1, a_2, \dots, a_l, a_{k + 1}, a_{k + 2}, a_{k + 3}\}$, where $l \leq k$. Define the maps $h_A \colon \M_A \to \C^k - \L_k$ and $h_B \colon \M_B \to \C^l - \L_l$ as follows:
$$
    h_A([\varphi]_{A}) = (\varphi(a_1), \varphi(a_2), \dots, \varphi(a_k)),\hspace{1.5cm}h_B([\psi]_{B}) = (\psi(a_1), \psi(a_2), \dots, \psi(a_l)).
$$
where the representatives $\varphi \colon S^2 \to \widehat{\C}$ and $\psi \colon S^2 \to \widehat{\C}$ are chosen so that $\varphi(a_{k + 1}) = \psi(a_{k + 1}) =~0, \varphi(a_{k + 2}) = \psi(a_{k + 2}) = 1$, and $\varphi(a_{k + 3}) = \psi(a_{k + 3}) = \infty$.

Then it easy to verify that diagram \eqref{fig: marked points forgetting} commutes, where $\mathrm{proj}_{k, l} \colon \C^k \to \C^l$ is the projection onto the first $l$ coordinates.
\begin{figure}[h]
\begin{tikzcd}
    \T_A \arrow[r, "t_{A, B}"] \arrow[d, "\pi_A"] & \T_B \arrow[d, "\pi_B"] \\
    \M_A \arrow[r, "m_{A, B}"] \arrow[d, "h_A"] & \M_B \arrow[d, "h_B"]\\
    \C^k - \L_k \arrow[r, "\mathrm{proj}_{k, l}"]& \C^l - \L_l
\end{tikzcd}
\caption{Maps induced by forgetting marked points.}\label{fig: marked points forgetting}
\end{figure}

The previous discussion, along with diagram \eqref{fig: marked points forgetting}, indicates that both $t_{A, B}$ and $m_{A, B}$ are holomorphic submersions. In particular, according to the submersion lemma (see, for instance, \cite[Corollary 5.13]{Lee}), if $\tau \in \T_B$, then its fiber $t_{A, B}^{-1}(\tau)$ is a properly embedded complex submanifold of the Teichm\"uller space $\T_A$ and $\dim_{\C}(t_{A, B}^{-1}(\tau)) = |A| - |B|$. Analogously, if $\nu \in \M_B$, then its fiber $m_{A, B}^{-1}(\nu)$ is a properly embedded complex submanifold of $\M_A$, also with complex dimension $|A| - |B|$.

When the sets $A$ and $B$ differ by only one point, a more refined result can be obtained from \textit{Bers Isomorphism}; see \cite{non_isometric_disks}.

\begin{proposition}\label{prop: fibers}
    Suppose that $A$ and $B$ are finite sets such that $B \subset A \subset S^2$, $|B| \geq 3$, and $|A| = |B| + 1$. If $\tau \in \T_B$, then $t_{A, B}^{-1}(\tau)$ is a properly and holomorphically embedded unit disk in the Teichm\"uller space $\T_A$. 
\end{proposition}

\subsection{Pullback maps}\label{subsec: pullback maps}
In this section we illustrate how the notions introduced in Section~\ref{subsec: teichmuller and moduli spaces} can be applied for studying the properties of Thurston maps. Most importantly, we introduce the following crucial concept.

\begin{proposition}\label{prop: def of sigma map}
    Suppose that $f \colon S^2 \dto S^2$ is finite-type topologically holomorphic map of admissible type, and $A$ and $B$ are two finite subsets of $S^2$ such that $|A|, |B| \geq 3$ and $B \subset \overline{f^{-1}(A)}$. Let $\varphi \colon S^2 \to \widehat{\C}$ be an orientation-preserving homeomorphism. Then there exists an orientation-preserving homeomorphism $\psi\colon S^2 \to \widehat{\C}$ such that $g_\varphi := \varphi \circ f \circ \psi^{-1}\colon \widehat{\C} \dto \widehat{\C}$ is holomorphic. In other words, the following diagram commutes
    $$
    \begin{tikzcd}
        (S^2, B) \arrow[r,"\psi"] \arrow[d,"f", dashed] & (\widehat{\C}, \psi(B)) \arrow[d,"g_\varphi", dashed]\\
     (S^2, A) \arrow[r,"\varphi"] & (\widehat{\C}, \varphi(A))
    \end{tikzcd}
    $$
    The homeomorphism $\psi$ is unique up to post-composition with a M\"obius transformation. Different choices of $\varphi$ that represent the same point in $\T_A$ yield maps $\psi$ that represent the same point in $\T_B$. 
    
    In other words, we have a well-defined map $\sigma_{f, A, B} \colon \T_A \to \T_B$ such that $\sigma_{f, A, B}([\varphi]_A) = [\psi]_B$, called the pullback map (or the $\sigma$-map). As $\varphi$ ranges across all maps representing a single point in $\T_A$, the map $g_\varphi$ is uniquely defined up to pre- and post-composition with M\"obius transformations.
\end{proposition}

\begin{proof}
    The existence of a homeomorphism $\psi \colon S^2 \to \widehat{\C}$, as well as its uniqueness up to pre-composition by a M\"obius transformation, is guaranteed by Lemma \ref{lemm: extensions} and Proposition \ref{prop: admissible type}. If we modify $\varphi$ by an isotopy rel.\ $A$, then $\psi$ changes only by an isotopy rel.\ $\overline{f^{-1}(A)} \supset B$ according to Corollary~\ref{corr: isotopy lifting property}. Post-composition of $\varphi$ by a M\"obius transformation does not affect the homeomorphism $\psi$. Thus, varying $\varphi$ within its equivalence class in $\T_A$ does not affect $[\psi]_{B}$, showing that the pullback map $\sigma_{f, A, B}$ introduced above is well-defined. These arguments also show that $g_{\varphi}$ is uniquely determined up to pre- and post-composition with a M\"obius transformation.
\end{proof}

\begin{remark}
    In the more classical convention, the roles of the sets $A$ and $B$ are usually reversed, meaning the domain of the pullback map is the Teichm\"uller space~$\mathcal{T}_A$, while the range is the Teichm\"uller space~$\mathcal{T}_B$; see~\cite{Pullback_invariants} and~\cite[Definition 3.1]{Astorg}. However, in this paper, we adopt the opposite convention.

    Moreover, except for this section (i.e., Section~\ref{subsec: pullback maps}), the sets $B$ and $C$ are always considered subsets of $A$, unless stated otherwise.
\end{remark}

Analogously to \cite[Proposition 2.16]{small_postsingular_set}, we can easily derive the following property related to the behavior of pullback maps.

\begin{proposition}\label{prop: dependence}
    Suppose that we are in the setting of Proposition \ref{prop: def of sigma map}, and there exist subsets $C \subset A$ and $D \subset B$ with $S_f \subset C$, $D \subset \overline{f^{-1}(C)}$, and $|C|, |D| \geq 3$. Then, the map~$g_\varphi$ depends only on the isotopy class rel.\ $C$ of $\varphi$ and $\psi|D$, and the equivalence class $[\psi]_{D}$ depends only on $[\varphi]_{C}$, i.e., the following diagram commutes:
    $$
    \begin{tikzcd}
        \T_A \arrow[rr,"\sigma_{f, A, B}"] \arrow[dd,"t_{A, C}"] & & \T_B \arrow[dd,"t_{B, D}"]\\
        & & \\
     \T_C \arrow[rr,"\sigma_{f, C, D}"] & & \T_D
    \end{tikzcd}
    $$
\end{proposition}

Using Proposition \ref{prop: composition}, it is straightforward to verify that the following composition rule holds in the context of pullback maps.

\begin{proposition}\label{prop: functoriality}
    Let $f_1 \colon S^2 \dto S^2$ and $f_2 \colon S^2 \dto S^2$ be two finite-type topologically holomorphic maps of admissible type. Suppose that $A$, $B$, and $C$ are finite subsets of $S^2$, each containing at least three points, and satisfying $S_{f_1} \subset A$, $S_{f_2} \subset B$, $B \subset \overline{f_1^{-1}(A)}$, and $C \subset \overline{f_2^{-1}(B)}$. Then $\sigma_{f_1 \circ f_2, C, A} = \sigma_{f_2, B, C} \circ \sigma_{f_1, A, B}$.
\end{proposition}

In this paper, the primary focus is on the case when $A = B$. In this scenario, either $f \colon (S^2, A) \rto$ is a Thurston map or $f$ is an orientation-preserving homeomorphism such that $f(A) = A$. We will use the notation $\sigma_{f, A}$ for the pullback map $\sigma_{f, A, A}$. Note that the domain and co-domain of $\sigma_{f, A}$ are the same, allowing us to explore its dynamical properties. 

The following observation provides the most crucial property of pullback maps.

\begin{proposition}\label{prop: fixed point of sigma}
    A Thurston map $f\colon (S^2, A) \righttoleftarrow$ with $|A| \geq 3$ is realized if and only if the pullback map~$\sigma_{f, A}$ has a fixed point in the Teichm\"uller space $\T_A$.
\end{proposition}

\begin{proof}
    Suppose that $f \colon (S^2, A) \rto$ is realized by a postsingularly finite holomorphic map $g\colon (\widehat{\C}, P) \rto$. Then, as noted in Remark~\ref{rem: comb equiv}, there exist orientation-preserving homeomorphisms $\varphi, \psi \colon S^2 \to \widehat{\C}$ such that $\varphi(A) = \psi(A) = P$, $\varphi$ and $\psi$ are isotopic rel.\ $A$, and $\varphi \circ f = g \circ \psi$. Clearly, $\tau = [\varphi] = [\psi] \in \T_A$ is a fixed point of $\sigma_{f, A}$.

    Now, suppose that $\tau = [\varphi] \in \T_A$ is a fixed point of $\sigma_{f, A}$. By Proposition \ref{prop: def of sigma map}, there exists an orientation-preserving homeomorphism $\psi \colon S^2 \to \widehat{\C}$ such that $g_\varphi := \varphi \circ f \circ \psi^{-1}\colon \widehat{\C} \dto \widehat{\C}$ is holomorphic and $[\varphi] = [\psi]$. Therefore, there exists a M\"obius transformation $M$ such that $\varphi$ and $M \circ \psi$ are isotopic rel.\ $A$ and, in particular, $\varphi|A = M \circ \psi|A$. Hence, the holomorphic map $g := g_{\varphi} \circ M^{-1}\colon (\widehat{\C}, \varphi(A)) \righttoleftarrow$ is postsingularly finite and combinatorially equivalent to~$f \colon (S^2, A) \rto$.  
\end{proof}

We say that a Thurston map $f\colon (S^2, A) \rto$ is a \textit{$(2, 2, 2, 2)$-map} if $f$ has a finite degree, exactly four postcritical values, and each of its critical points is \textit{simple} (i.e., has local degree~2) and strictly pre-periodic. The following proposition summarizes the most important properties of pullback maps.

\begin{proposition}\label{prop: propeties of pullback map}
    Let $f \colon (S^2, A) \rto$ be a Thurston map with $|A| \geq 3$. Then 
    \begin{enumerate}
        \item \label{it: hol} $\sigma_{f, A}$ holomorphic;
        \item \label{it: general} $\sigma_{f, A}$ is \textit{1-Lipschitz}, i.e., $d_T(\sigma_{f, A}(\tau_1), \sigma_{f, A}(\tau_2)) \leq d_T(\tau_1, \tau_2)$ for every $\tau_1, \tau_2 \in \T_A$; 
        \item \label{it: transc case} $\sigma_{f, A}$ is \textit{distance-decreasing}, i.e., $d_T(\sigma_{f, A}(\tau_1), \sigma_{f, A}(\tau_2)) < d_T(\tau_1, \tau_2)$ for every distinct $\tau_1, \tau_2 \in \T_A$, if $f$ is transcendental;
        \item \label{it: not 2,2,2,2 case} $\sigma_{f, A}^{\circ k}$ is distance-decreasing for some $k \geq 1$ if $f$ is not a $(2, 2, 2, 2)$-map;
        \item \label{it: 2,2,2,2 case} $\sigma_{f, P_f}$ is an isometry of $\T_{P_f}$ if $f$ is a $(2, 2, 2, 2)$-map.
    \end{enumerate}
\end{proposition}
\begin{proof}
    This result is rather well-known in the context of finite degree Thurston maps (see \cite[Section 1.3]{Buff} and \cite[Sections 10.6 and 10.7, and Appendix C.6]{Hubbard_Book_2}), and it can be extended analogously to our setting. We provide only an outline of the argument here, leaving some of the details to the reader.
    
    According to Lemma \ref{lemm: extensions} and Proposition \ref{prop: admissible type}, there exist two orientation preserving homeomorphisms $\varphi, \psi \colon S^2 \to \widehat{\C}$ and a holomorphic map $g \colon \widehat{\C} \dto \widehat{\C}$ so that $f = \varphi^{-1} \circ g \circ \psi$. Let $B := \varphi(A)$ and $C := \psi(A)$. In this proof, we identify the topological sphere $S^2$ with the Riemann sphere~$\widehat{\C}$. By Proposition \ref{prop: functoriality}, we have the following:
    $$
        \sigma_{f, A} = \sigma_{\psi, C, A} \circ \sigma_{g, B, C} \circ \sigma_{\varphi^{-1}, A, B}.
    $$
    It is well-known that $\sigma_{\psi, C, A}\colon \T_C \to \T_A$ and $\sigma_{\varphi^{-1}, A, B} \colon \T_A \to \T_B$ are holomorphic isometries (see \cite[Sections 6.4 and 7.4]{Hubbard_Book_1}). Using an approach analogous to \cite[Lemma 3.3]{Astorg}, one can show that $\sigma_{g, B, C} \colon \T_B \to \T_C$ is also holomorphic, which establishes item \eqref{it: hol}. Similarly, according to \cite[Lemma 16]{Adam_Thesis} or \cite[Sections 2 and 3]{Astorg}, the pullback map $\sigma_{g, B, C}$ must be distance-decreasing if the map $f$ is transcendental, leading to item \eqref{it: transc case}. Item \eqref{it: general} follows from \cite[Corollary 6.10.7]{Hubbard_Book_1}, and item \eqref{it: not 2,2,2,2 case} comes from the previous discussion and \cite[Lemma 2.9]{Buff} (see also \cite[Corollary 10.7.8]{Hubbard_Book_2}). Lastly, item \eqref{it: 2,2,2,2 case} is well-known; see \cite[Section 9]{DH_Th_char} or \cite[Appendix C.6]{Hubbard_Book_2}.
\end{proof}

\begin{remark}\label{rem: remark on 2,2,2,2-maps}
    Suppose that $f$ is a $(2, 2, 2, 2)$-map. According to item \eqref{it: 2,2,2,2 case}, the pullback map $\sigma_{f, P_f}$ is an automorphism of the Teichm\"uller space $\T_{P_f}$, which is biholomorphic to the unit disk, as discussed in Section \ref{subsec: teichmuller and moduli spaces}. Moreover, it is known that there exist examples of Thurston maps $f$ for which $\sigma_{f, P_f}$ is the identity, or it is conjugate to an elliptic, parabolic, or hyperbolic transformation of the unit disk; see \cite[Appendix C.6]{Hubbard_Book_2}.
\end{remark}

Proposition \ref{prop: propeties of pullback map}, called \textit{Thurston's rigidity} (cf. \cite[Corollary 10.7.8]{Hubbard_Book_2}), illustrates how the properties of a pullback map can be used to derive conclusions about dynamics of the corresponding Thurston map.

\begin{proposition}\label{prop: rigidity}
    Suppose that $f\colon (S^2, A) \rto$ is not a $(2, 2, 2, 2)$-map. If $f\colon (S^2, A) \rto$ is realized by two postsingularly finite holomorphic maps $g_1 \colon (\widehat{\C}, P_1) \rto$ and $g_2 \colon (\widehat{\C}, P_2) \rto$, then there exists a M\"obius transformation $M$ such that $M \circ g_1 = g_2 \circ M$ and $M(P_1) = P_2$.
\end{proposition}

\begin{proof}
    According to Remark \ref{rem: comb equiv}, there exist orientation-preserving homeomorphisms $\varphi_1$, $\varphi_2$, $\psi_1$, $\psi_2 \colon S^2 \to \widehat{\C}$ such that $\varphi_i$ and $\psi_i$ are isotopic rel.\ $A$, $\psi_i(A) = \varphi_i(A) = P_i$, and $g_i = \varphi_i \circ f \circ \psi_i^{-1}$ for each $i = 1, 2$. Clearly, if $|A| \geq 3$, the points $\tau_i = [\varphi_i] = [\psi_i]$, $i = 1, 2$ are fixed points of the pullback map $\sigma_{f, A}$. However, item \eqref{it: not 2,2,2,2 case} of Proposition \ref{prop: propeties of pullback map} demonstrates that $\sigma_{f, A}$ can have at most one fixed point in the Teichm\"uller space $\T_A$. Therefore, $\tau_1 = \tau_2$, and there exists a M\"obius transformation $M$ such that $\varphi_2$ is isotopic rel.\ $A$ to $M \circ \varphi_1$. Consequently, $\psi_2$ must isotopic rel.\ $A$ to $M \circ \psi_1$ and $M(P_1) = P_2$. Corollary \ref{corr: isotopy lifting property} then implies that $g_1$ and~$g_2$ must be conjugate by $M$. If $|A| = 2$, the same conclusion follows from \cite[Proposition 2.3]{farb_margalit}.
\end{proof}

\subsection{Obstructions}\label{subsec: obstructions}

Let $A$ be a finite subset of $S^2$. We say that a simple closed curve $\gamma \subset S^2 - A$ is \textit{essential} in $S^2 - A$ if each connected component of $S^2 - \gamma$ contains at least two points of the set $A$; otherwise $\gamma$ is called \textit{non-essential}. In other words, $\gamma$ is essential in $S^2 - A$ if it cannot be shrinked to a point via a~homotopy~in~$S^2 - A$. It is easy to see that this property is independent from the homotopy class of the curve in $S^2 - A$. Finally, we note that essential simple closed curves in $S^2 - A$ exist if and only if $|A| \geq 4$.

Now, let $f \colon S^2 \dto S^2$ be a topologically holomorphic map of finite type, and let $A$ and $B$ be two finite sets with $S_f \subset A$ and $B \subset \overline{f^{-1}(A)}$. Suppose that $\gamma \subset S^2 - A$ and $\widetilde{\gamma} \subset S^2 - \overline{f^{-1}(A)}$ are two simple closed curves such that $\widetilde{\gamma} \subset f^{-1}(\gamma)$. If $\gamma$ is non-essential in $S^2 - A$, then $\widetilde{\gamma}$ is also non-essential in $S^2 - B$ (see, for instance, \cite[Propositions~2.11 and~2.12]{our_approx}). However, if $\gamma$ is essential in $S^2 - A$, then $\widetilde{\gamma}$ may or may not be essential in $S^2 - B$. A \textit{multicurve} $\mathcal{G}$ in $S^2 - A$ is a finite collection of essential simple closed curves in $S^2 - A$ that are pairwise disjoint and pairwise non-homotopic in $S^2 - A$.

\begin{definition}\label{def: levy cycle}
    Let $f \colon (S^2, A) \rto$ be a Thurston map. A collection $\mathcal{G} = \{\gamma_1, \gamma_2, \dots, \gamma_r\}$ of essential simple closed curves in $S^2 - A$ is called a \textit{Levy cycle} for the Thurston map $f\colon (S^2, A) \rto$ if there exists a cyclic permutation $\sigma$ of $\{1, 2, \dots, r\}$ so that for every $i$ there is a simple closed curve $\widetilde{\gamma}_i \in f^{-1}(\gamma_{i})$ that is homotopic in $S^2 - A$ to $\gamma_{\sigma(i)}$ and $\deg(f|\widetilde{\gamma}_i \colon \widetilde{\gamma}_i \to \gamma_i) = 1$. 
    A Levy cycle $\mathcal{G}$ is \textit{degenerate} if each $\gamma_i$ bounds an open Jordan region $D_i$ so that for every~$i$ some connected component $\widetilde{D}_i$ of $f^{-1}(D_i)$ is homotopic in $S^2 - A$ to $D_{\sigma(i)}$ and maps with degree~$1$ by the map $f$ to $D_i$. 
\end{definition}

The following observation is widely known in the context of finite degree Thurston maps \cite[Exercise 10.3.6]{Hubbard_Book_2}, and its proof extends to the case of transcendental Thurston maps as well \cite[Proposition 2.11]{small_postsingular_set}.

\begin{proposition}\label{prop: levy cycles are obstructions}
    A Thurston map $f\colon (S^2, A) \rto$ having a Levy cycle is obstructed.
\end{proposition}

If $\mathcal{G} = \{\gamma_1, \gamma_2, \dots, \gamma_r\}$ is a Levy cycle for a Thurston map $f \colon (S^2, A) \rto$, then the number~$r$ is called the \textit{period} of the Levy cycle $\mathcal{G}$. We say that the Levy cycle $\mathcal{G}$ is \textit{minimal}, if the curves in $\mathcal{G}$ are pairwise non-homotopic in $S^2 - A$. Clearly, if $\mathcal{G}$ is not minimal, we can extract $\mathcal{G}' \subset \mathcal{G}$ that provides a minimal Levy cycle for $f \colon (S^2, A) \rto$.

We observe that if $\mathcal{G} = \{\gamma_1, \gamma_2, \dots, \gamma_r\}$ is a Levy cycle for a Thurston map $f \colon (S^2, A) \rto$ and $\mathcal{G}' = \{\gamma_1', \gamma_2', \dots, \gamma_r'\}$ is a collection of simple closed curves in $S^2 - A$ such that $\gamma_i$ and $\gamma_i'$ are homotopic in $S^2 - A$ for $i=1, 2, \dots, r$, then by Corollary \ref{corr: homotopy lifting for curves}, $\mathcal{G}'$ is also a Levy cycle for $f \colon (S^2, A) \rto$. Furthermore, if $\mathcal{G}$ is degenerate, then according to the second part of Corollary \ref{corr: homotopy lifting for curves}, $\mathcal{G}'$ must also be degenerate.

If additionally $\mathcal{G}$ is a multicurve, we say that $\mathcal{G}$ is a \textit{Levy multicurve}. When the period of a Levy cycle $\mathcal{G}$ equals 1, then $\mathcal{G}$ consists of a single essential simple closed curve $\gamma$ in $S^2 - A$ that is called a \textit{Levy fixed} or \textit{Levy invariant curve}. When the Levy cycle $\mathcal{G}$ of period 1 is degenerate, then $\gamma$ is called \textit{degenerate} Levy fixed (or invariant) curve.

It is known that a Levy cycle, or even a Levy fixed curve, is not necessarily degenerate (see, for instance, \cite[Section 3]{critically_fixed} for a large class of such examples). At the same time, it is easy construct an example of a degenerate Levy cycle $\mathcal{G}$ for a Thurston map $f \colon (S^2, A) \rto$ that does not form a Levy multicurve even after replacing curves in $\mathcal{G}$ with simple closed curves homotopic to them in $S^2 - A$.

\section{Marked Thurston maps}\label{sec: marked Thurston maps}

In this section, we investigate the realizability question and Levy cycles for marked Thurston maps as introduced in Section \ref{subsec: thurston maps}. In Section \ref{subsec: detecting levy cycles}, we present tools for detecting and analyzing Levy cycles. Further, in Sections \ref{subsec: forgetting marked points}-\ref{subsec: thurston maps with extra marked points}, we explore how the properties of a Thurston map $f \colon (S^2, A) \rto$ are related to those of the Thurston map $f \colon (S^2, B) \rto$, where $B \subset A$ is an $f$-pseudo-invariant subset containing $S_f$. Also, in Section \ref{subsec: passing to an iterate}, we investigate the relationship between the properties of a Thurston map and those of its iterates. Main Theorem~\ref{mainthm A} is proven in Section~\ref{subsec: thurston maps with extra marked points}, while Propositions~\ref{prop: first prop intro} and~\ref{prop: second prop intro} from the introduction are established in Sections~\ref{subsec: extra fixed point} and ~\ref{subsec: passing to an iterate}, respectively.

\subsection{Detecting Levy cycles}\label{subsec: detecting levy cycles}

Let $X$ be a hyperbolic Riemann surface, and let $\alpha \subset X$ be a simple closed $C^1$-curve. The length of $\alpha$ with respect to the hyperbolic metric on~$X$ is denoted by $\ell_{X}(\alpha)$. Now, let $\alpha$ be an essential simple closed curve in the punctured Riemann sphere $X := \widehat{\C} - P$, where $3 \leq |P| < \infty$. According to \cite[Proposition 3.3.8]{Hubbard_Book_1}, there exists a unique closed geodesic in $X$ that is homotopic (in $X$) to $\alpha$. Note that this geodesic should be necessarily simple \cite[Proposition 3.3.9]{Hubbard_Book_1}. 

Let $\gamma$ be an essential simple closed curve in $S^2 - A$, and $\tau = [\varphi]$ be a point in the Teichm\"uller space $\T_A$. We define $l_{\gamma}(\tau)$ as the length of the unique hyperbolic geodesic in $\widehat{\C} - \varphi(A)$ that is homotopic in $\widehat{\C} - \varphi(A)$ to $\varphi(\alpha)$. Additionally, we introduce $w_{\gamma}(\tau) := \log l_{\gamma}(\tau)$. It is known that $w_{\gamma} \colon \T_A \to  \R$ is a $1$-Lipschitz function \cite[Theorem~7.6.4]{Hubbard_Book_1}.

Let $X$ be a hyperbolic Riemann surface, and suppose $\alpha \subset X$ is a simple closed geodesic with $\ell_{X}(\alpha) < \ell^*$, where $\ell^* := \log(3 + 2 \sqrt{2})$. In this case, we say that $\alpha$ is \textit{short}. As stated in \cite[Proposition 3.3.8 and Corollary~3.8.7]{Hubbard_Book_1}, two short simple closed geodesics on a hyperbolic Riemann surface $X$ are either disjoint and non-homotopic in $X$, or they coincide. Therefore, \cite[Proposition 3.3.8]{Hubbard_Book_1} implies that a punctured Riemann sphere $\widehat{\C} - P$, where $3 \leq |P| < \infty$, can have at most $|P| - 3$ distinct short simple closed geodesics.


The following result allows us to identify a Levy cycle for a Thurston map based on the behavior of the corresponding pullback map.

\begin{proposition}\label{prop: finding a levy cycle}
    Suppose that $f \colon (S^2, A) \rto$ is a Thurston map with $|A| = k + 3, k \geq 1$. Let $\tau_1 \in \T_A$ and $\tau_i = [\varphi_i] = \sigma_{f, A}^{\circ (i - 1)}(\tau)$ for $i = 1, 2, \dots, k + 1$, where the representatives $\varphi_i \colon S^2 \to \widehat{\C}$ are chosen so that the map $g_i := \varphi_i \circ f \circ \varphi_{i + 1}^{-1}\colon \widehat{\C} \dto \widehat{\C}$ is holomorphic for each $i = 1, 2 \dots, k$. Assume that there exists an annulus $U \subset \widehat{\C}$ such that
    \begin{itemize}
        \item each connected component of $\widehat{\C} - U$ contains at least two points of $\varphi_{k + 1}(A)$;
    
        \item $\mod(U) > (k + 4) \pi e^{k d_0} / \ell^*$, where $d_0 := d_T(\tau_1, \tau_2)$;
        
        \item $g_1 \circ g_2 \circ \dots \circ g_k$ is defined and injective on $U$.
    \end{itemize}
    Then the Thurston map $f \colon (S^2, A) \rto$ has a Levy multicurve $\mathcal{G}$. Moreover, for each $\gamma \in \mathcal{G}$, the inequality $l_{\gamma}(\tau_i) < (k + 4)\pi e^{k d_0} / \mathrm{mod}(U)$ holds for each $i = 1, 2, \dots, k + 1$.
\end{proposition}

\begin{proof}
    Define $U_1 := U$ and $U_{i + 1} := g_{k - i + 1}(U_i)$ for each $i = 1, 2, \dots, k$. Notice that each $U_i$ is an annulus with $\mod(U_i) = \mod(U)$. Therefore, we can find a parallel subannulus $V$ of $U_{k + 1}$ such that $\mod(V) \geq \mod(U) / (k + 4)$ and $V$ does not contain any points of $\varphi_1(A)$. Next, define $V_1 := V$, and for each $i=1, 2, \dots, k$, let $V_{i + 1}$ be the annulus coinciding with the connected component of $g_i^{-1}(V_i)$ within $U_{k - i + 1}$. Clearly, $V_i \cap \varphi_i(A) = \emptyset$ for each $i = 1, 2, \dots, k + 1$. It also can be easily verified that each connected component of $\widehat{\C} - V_i$ contains at least two points of $\varphi_i(A)$; otherwise, there would exist a connected component of $\widehat{\C} - U$ containing at most one point of~$\varphi_{k + 1}(A)$ (see, for instance, \cite[Theorems 5.10 and 5.11]{Forster}).

    Let $\alpha_{k + 1}$ be a unique hyperbolic geodesic of $V_{k + 1}$. It is known that $\alpha_{k + 1}$ is a simple closed geodesic that forms a core curve of the annulus $V_{k + 1}$, and its length in $V_{k + 1}$ is given by $\ell_{V_{k + 1}}(\alpha_{k + 1}) = \pi / \mod(V_{k + 1})$; see \cite[Proposition 3.3.7]{Hubbard_Book_1} For each $i = 2, 3, \dots, k + 1$, let $\alpha_{i - 1} := g_{i - 1}(\alpha_i)$. Next, define $\beta_i := \varphi_i^{-1}(\alpha_i)$ for each $i = 1, 2, \dots, k + 1$. It is straightforward to verify that $\beta_i$ is an essential simple closed curve in $S^2 - A$, with $\beta_{i - 1} = f(\beta_i)$ and $\deg(f|\beta_i \colon \beta_i \to \beta_{i - 1}) = 1$; see Figure \ref{fig: huge diag} for further clarifications.
    \begin{figure}[t]
    \begin{tikzcd}
        \beta_{k + 1} \arrow[r, hook] \arrow[d, "f"] & S^2 \arrow[d,"f", dashed] \arrow[r, "\varphi_{k + 1}"] & S^2 \arrow[d,"g_k", dashed] & U_1 \arrow[l, hook] \arrow[d,"g_k"] & V_{k + 1}  \arrow[d,"g_k"] \arrow[l, hook] & \alpha_{k + 1} \arrow[l, hook]  \arrow[d,"g_k"] \\
        \beta_k \arrow[r, hook] \arrow[dd, dotted] & S^2 \arrow[r, "\varphi_k"] \arrow[dd, dotted] & S^2 \arrow[dd, dotted] & U_2 \arrow[l, hook] \arrow[dd, dotted] & V_k \arrow[l, hook] \arrow[dd, dotted] & \alpha_k \arrow[l, hook] \arrow[dd, dotted] \\
        & & & & & \\
        \beta_3  \arrow[r, hook] \arrow[d, "f"] & S^2  \arrow[d,"f", dashed] \arrow[r, "\varphi_3"] & S^2 \arrow[d,"g_2", dashed] & U_{k - 1} \arrow[l, hook] \arrow[d,"g_2"] & V_3 \arrow[l, hook]  \arrow[d,"g_2"] & \alpha_3 \arrow[l, hook]  \arrow[d,"g_2"] \\
        \beta_2 \arrow[r, hook] \arrow[d, "f"] & S^2 \arrow[d,"f", dashed] \arrow[r, "\varphi_2"] & S^2 \arrow[d,"g_1", dashed] & U_k \arrow[l, hook]  \arrow[d,"g_1"] & V_2 \arrow[l, hook]  \arrow[d,"g_1"] & \alpha_2 \arrow[l, hook]  \arrow[d,"g_1"] \\
        \beta_1 \arrow[r, hook]  & S^2 \arrow[r, "\varphi_1"] & S^2 & U_{k + 1} \arrow[l, hook] & V_1 \arrow[l, hook] & \alpha_1 \arrow[l, hook]
    \end{tikzcd}
    \caption{Diagram depicting the relationships among the introduced~objects.}\label{fig: huge diag}
    \end{figure}
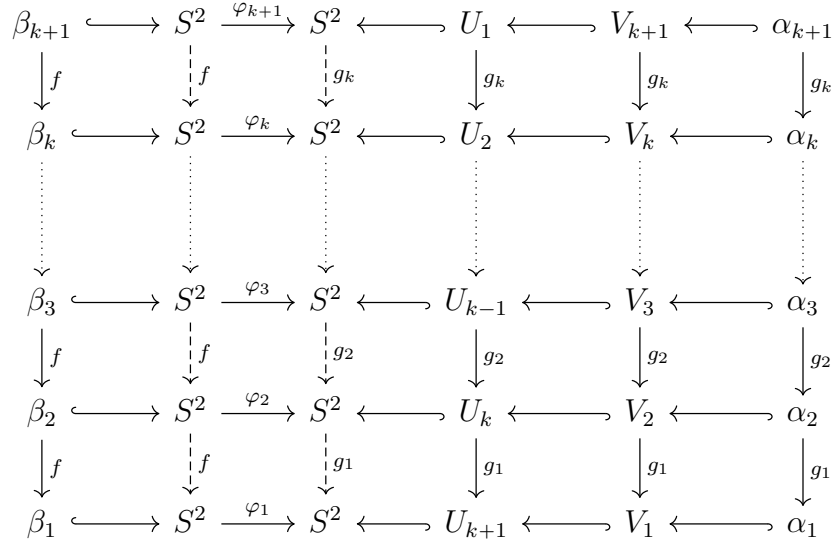

    Let $X_i := \widehat{\C} - \varphi_i(A)$ for each $i = 1, 2, \dots, k + 1$. Recall that $g_i|V_{i + 1} \colon V_{i + 1} \to V_i$ is a biholomorphism for each $i = 1, 2, \dots, k$, and therefore, by Schwarz-Pick's lemma, we have $\ell_{V_i}(\alpha_i) = \ell_{V_j}(\alpha_j)$ for all $i$ and $j$. Thus, according to Schwarz-Pick's lemma and the choice of $\alpha_{k + 1}$, we have: 
    $$
        \ell_{X_i}(\alpha_i) \leq \ell_{V_i}(\alpha_i) = \ell_{V_{k + 1}}(\alpha_{k + 1}) = \frac{\pi}{\mod(V_{k + 1})} = \frac{\pi}{\mod(V)} \leq \frac{(k + 4) \pi}{\mod(U)}.
    $$

    By item \eqref{it: general} of Proposition \ref{prop: propeties of pullback map}, we have $d_T(\tau_i, \tau_j) \leq |i - j| \cdot d_0$ for each $i, j = 1, 2, \dots, k+1$. Therefore, since the function $w_{\beta_i} \colon \T_A \to \R$ is 1-Lipschitz, it follows that $l_{\beta_i}(\tau_j) < (k + 4) \pi e^{k d_0} / \mod(U)$. In particular, there exist short simple closed geodesics $\delta_1, \delta_2, \dots, \delta_{k + 1}$ in $X_1$, with each $\delta_i$ homotopic in $X_1$ to $\varphi_1(\beta_i)$.
    
    As it was discussed previously, any two short geodesics, $\delta_i$ and $\delta_j$, are either disjoint and non-homotopic in $X_1$, or they coincide. However, the punctured Riemann sphere $X_1 = \widehat{\C} - \varphi_1(A)$ can contain at most $k$ distinct short simple closed geodesics. Therefore, there must be distinct indices $i_1$ and $i_2$ such that $\delta_{i_1} = \delta_{i_2}$.
    Without loss of generality, assume that $i_2$ is the minimal index such that $i_2 > i_1$ and $\delta_{i_1} = \delta_{i_2}$. Since $\varphi_1^{-1}(\delta_i)$ is homotopic in $S^2 - A$ to $\beta_i$ for each $i = 1, 2, \dots k + 1$, Corollary \ref{corr: homotopy lifting for curves} implies that the multicurve $\{\varphi_1^{-1}(\delta_{i_1 + 1}), \varphi_1^{-1}(\delta_{i_1 + 2}), \dots, \varphi_1^{-1}(\delta_{i_2})\}$ forms a Levy multicurve for the Thurston map $f \colon (S^2, A) \rto$. Finally, note that $l_{\varphi_1^{-1}(\delta_i)}(\tau) = l_{\beta_i}(\tau)$ for each $i=1, 2,\dots, k + 1$ and $\tau \in \T_A$, and the rest follows from the previous discussion.
\end{proof}

The following proposition demonstrates that, under specific conditions, a Levy multicurve can be obtained from a Levy cycle by simply modifying the homotopy classes of the corresponding simple closed curves.

\begin{proposition}\label{prop: levy cycles to levy multicurves}
    Let $\mathcal{G} = \{\gamma_1, \gamma_2, \dots, \gamma_r\}$ be a minimal Levy cycle for a Thurston map $f \colon (S^2, A) \rto$. Suppose there exist a curve $\gamma \in \mathcal{G}$ and a point $\tau$ in the Teichm\"uller space~$\T_A$ such that $l_{\gamma}(\tau) < \ell^* e^{-(r - 1)d_0}$, where $d_0 := d_T(\tau, \sigma_{f, A}(\tau))$. Then, there exists a Levy multicurve $\mathcal{G}' = \{\gamma_1', \gamma_2', \dots, \gamma_r'\}$ for $f \colon (S^2, A) \rto$ such that each $\gamma_i'$ is homotopic in $S^2 - A$ to $\gamma_i$. Moreover, if $\mathcal{G}$ is a degenerate Levy cycle, then $\mathcal{G}'$ is a degenerate Levy multicurve.
\end{proposition}

\begin{proof}
    Without loss of generality, assume that $\gamma = \gamma_1$ and $\sigma(i) = i + 1\, \mod\, r$, where $\sigma$ is a cyclic permutation associated with the Levy cycle $\mathcal{G}$ (see Definition \ref{def: levy cycle}).
    
    Let $\tau_i = [\varphi_i] = \sigma_{f, A}^{\circ (i - 1)}(\tau)$ for $i \in \N$, where the representatives $\varphi_i \colon S^2 \to \widehat{\C}$ are chosen so that $g_i := \varphi_i \circ f \circ \varphi_{i + 1}^{-1} \colon \widehat{\C} \dto \widehat{\C}$ is holomorphic. Additionally, define $X_i := \widehat{\C} - \varphi_i(A)$ and $Y_i := \widehat{\C} - \overline{g_i^{-1}(\varphi_i(A))} = \widehat{\C} - \overline{\varphi_{i + 1}(f^{-1}(A))}$ for all $i \in \N$.
    \begin{claim}
        For each $i = 1, 2, \dots, r$, we have the inequality $l_{\gamma_i}(\tau_i) \leq l_{\gamma_1}(\tau_1)$.
    \end{claim}
    \begin{subproof}
        We assume that $r \geq 2$, and begin by showing that $l_{\gamma_{2}}(\tau_2) \leq l_{\gamma_1}(\tau_1)$. The rest then follows by induction.
        There exists a simple closed geodesic $\alpha_1$ in $X_1$ such that $\ell_{X_1}(\alpha_1) = l_{\alpha_1}(\tau_1)$ and $\alpha_1$ is homotopic in $X_1$ to $\varphi_1(\gamma_1)$. Denote $\beta_1 := \varphi_1^{-1}(\alpha_1)$. Since $\beta_1$ is homotopic in $S^2 - A$ to $\gamma_1$, by Corollary~\ref{corr: homotopy lifting for curves}, there exists a simple closed curve $\beta_2 \subset f^{-1}(\beta_1)$ that is homotopic in $S^2 - A$ to $\gamma_2$ and satisfies $\deg(f|\beta_2 \colon \beta_2 \to \beta_1) = 1$. Let $\alpha_2 := \varphi_2(\beta_2) \subset Y_1$. Obviously, $g_1(\alpha_2) = \alpha_1$ and $\deg(g_1|\alpha_2 \colon \alpha_2 \to \alpha_1) = 1$. Therefore, since $g_1|Y_1 \colon Y_1 \to X_1$ is a holomorphic covering map, Schwarz-Pick's lemma gives: 
        $$
            \ell_{X_2}(\alpha_2) \leq \ell_{Y_1}(\alpha_2) = \ell_{X_1}(\alpha_1) = l_{\gamma_1}(\tau_1).
        $$
        At the same time, clearly, $l_{\gamma_2}(\tau_2) \leq \ell_{X_2}(\alpha_2)$, and the rest follows.
        \end{subproof}

        Item \eqref{it: general} of Proposition \ref{prop: propeties of pullback map} implies that $d_T(\tau_i, \tau_j) \leq |i - j| \cdot d_0$ for each $i, j \in \N$. Therefore, based on the claim above and the fact that $w_{\gamma_i} \colon \T_A \to \R$ is $1$-Lipschitz, we conclude that $l_{\gamma_i}(\tau_1) < \ell^*$ for each $i = 1, 2, \dots, r$. Therefore, there exist short simple closed geodesics $\delta_1, \delta_2, \dots, \delta_r$ in $X_1$, where each $\delta_i$ is homotopic in $X_1$ to $\varphi_1(\gamma_i)$. 
        
        Since $\mathcal{G}$ is a minimal Levy cycle, no two geodesics $\delta_i$ and $\delta_j$ are homotopic in $X_1$ for $i \neq j$. Therefore, as each $\delta_i$ is a short simple closed geodesic, $\delta_i$ and $\delta_j$ must be disjoint in $X_1$. Now, define $\gamma_i' := \varphi_1^{-1}(\delta_i)$ and $\mathcal{G}' := \{\gamma_1', \gamma_2', \dots, \gamma_r'\}$, which forms a multicurve in $S^2 - A$. Since $\gamma_i$ and $\gamma_i'$ are homotopic in $S^2 - A$, $\mathcal{G}'$ is a Levy multicurve for the Thurston map $f \colon (S^2, A) \rto$, and if the original Levy cycle $\mathcal{G}$ is degenerate, then $\mathcal{G}'$ is also degenerate.
\end{proof} 

\subsection{Trivial marked points}\label{subsec: forgetting marked points}

Let $f \colon (S^2, A) \rto$ be a Thurston map. We say that a marked point $a \in A$ is \textit{trivial} for $f \colon (S^2, A) \rto$ if $a \in A - P_f$ and either $a$ is strictly pre-periodic, or it eventually lands to an essential singularity of the map $f$ after several iterations (this includes the case when $a$ is already an essential singularity of $f$). It is easy to see that if $a \in A - P_f$ is not a trivial marked point of $f \colon (S^2, A) \rto$, then $a$ must be periodic under $f$. The following proposition demonstrates that adding trivial marked points to the marked set does not affect the essential properties of a Thurston map (cf. \cite[Proposition 2.8]{park2}).

\begin{proposition}\label{prop: trivial marked points}
    Let $f \colon (S^2, A) \rto$ be a Thurston map, where the set $A$ can be decomposed as $A = B \sqcup C$, with $S_f \subset B$, $B$ being $f$-pseudo-invariant, and $C$ consisting of trivial marked points. Then
    \begin{enumerate}
        \item \label{it: realizability and trivial marked points} $f\colon (S^2, B) \rto$ is realized if and only if $f \colon (S^2, A) \rto$ is realized. Moreover, if $|B| \geq 3$ and $\tau \in \T_B$ is a fixed point of the pullback map $\sigma_{f, B}$, then the pullback map $\sigma_{f, A}$ has a unique fixed point in $t^{-1}_{A, B}(\tau)$;

        \item \label{it: realizability and levy cycles} $f \colon (S^2, B) \rto$ has a Levy cycle (or degenerate Levy cycle, Levy multicurve, or degenerate Levy multicurve) $\mathcal{G} = \{\gamma_1, \gamma_2, \dots, \gamma_r\}$, then $f \colon (S^2, A) \rto$ has a Levy cycle (or degenerate Levy cycle, Levy multicurve, or degenerate Levy multicurve, respectively) $\widetilde{\mathcal{G}} = \{\widetilde{\gamma}_1, \widetilde{\gamma}_2, \dots, \widetilde{\gamma}_r\}$ such that $\gamma_i$ and $\widetilde{\gamma}_i$ homotopic in $S^2 - B$ for each $i = 1, 2, \dots, r$;

        \item \label{it: extra} if $\mathcal{G}$ is a Levy cycle (or a degenerate Levy cycle) for $f \colon (S^2, A) \rto$, then $\mathcal{G}$ is also a Levy cycle (degenerate Levy cycle, respectively) for $f \colon (S^2, B) \rto$.
    \end{enumerate}
\end{proposition}

\begin{proof}
    Let $m \geq 1$ be an integer such that every point of $C$ lands either to a point in $B$ within at most $m$ iterations, or to an essential singularity of $f$ after at most $m - 1$ iterations. It is straightforward to see that $A \subset \overline{f^{-m}(B)}$.

    For item \eqref{it: realizability and trivial marked points}, we can assume that $|B| \geq 3$. If $|B| \leq 2$, this can be easily adjusted by moving points from $C$ to $B$ and applying item \eqref{it: three points} of  Proposition \ref{prop: small set}. 
    If the Thurston map $f \colon (S^2, A) \rto$ is realized, then naturally $f \colon (S^2, B) \rto$ is also realized. Now, assume that $f \colon (S^2, B) \rto$ is realized, and the pullback map $\sigma_{f, B}$ has a fixed point $\tau \in   \T_{B}$. According to Propositions \ref{prop: dependence} and \ref{prop: functoriality}, it follows that $[\psi]_{A} = \sigma_{f, A}^{\circ m}([\varphi]_{A}) = \sigma_{f^{\circ m}, A}([\varphi]_{A})$ depends only on $[\varphi]_{B}$ for any $[\varphi]_{A} \in \T_A$, and that $[\psi]_{B} = [\varphi]_{B}$ if $[\varphi]_{A} \in t_{A, B}^{-1}(\tau)$. Therefore, $\sigma_{f, A}^{\circ m}(t_{A, B}^{-1}(\tau))$ consists of a single point $\mu \in t_{A, B}^{-1}(\tau)$, making $\mu$ a unique fixed point of $\sigma_{f, A}$ within $t_{A, B}^{-1}(\tau)$. In particular, Proposition~\ref{prop: fixed point of sigma} implies that $f \colon (S^2, A) \rto$ is realized.

    For items \eqref{it: realizability and levy cycles} and \eqref{it: extra}, we assume that $m = 1$, i.e., for every $c \in C$, either $f(c) \in B$ or $c$ is an essential singularity of $f$. Equivalently, this means that $A \subset \overline{f^{-1}(B)}$. The general case can be obtained through induction on $m$, and we leave this to the reader.
    
    Suppose $f \colon (S^2, B) \rto$ has a Levy cycle $\mathcal{G} = \{\gamma_1, \gamma_2, \dots, \gamma_r\}$, and let $\sigma$ be the corresponding cyclic permutation of the set $\{1, 2, \dots, r\}$. Then, for each $i = 1, 2, \dots, r$, there exists a simple closed curve $\widetilde{\gamma}_i \subset f^{-1}(\gamma_i)$ such that $\widetilde{\gamma}_i$ is homotopic in $S^2 - B$ to $\gamma_{\sigma(i)}$ and $\deg(f|\widetilde{\gamma}_i \colon \widetilde{\gamma}_i \to \gamma_i) = 1$. Now, consider the collection of essential simple closed curves $\widetilde{\mathcal{G}} := \{\widetilde{\gamma}_1, \widetilde{\gamma}_2, \dots, \widetilde{\gamma}_r\}$ in $S^2 - A$. We aim to show that $\widetilde{\mathcal{G}}$ provides a Levy cycle for the Thurston map $f \colon (S^2, A) \rto$. 
    
    There exists a simple closed curve $\widetilde{\gamma}_{\sigma(i)} \subset f^{-1}(\gamma_{\sigma(i)})$ such that $\deg(f|\widetilde{\gamma}_{\sigma(i)}\colon \widetilde{\gamma}_{\sigma(i)} \to \gamma_{\sigma(i)}) =~1$ for each $ i = 1, 2, \dots, r$. Since $\widetilde{\gamma}_i$ and $\gamma_{\sigma(i)}$ are homotopic in $S^2 - B$, Corollary~\ref{corr: homotopy lifting for curves} guarantees the existence of a simple closed curve $\delta_i \subset f^{-1}(\widetilde{\gamma}_i)$ such that $\delta_i$ is homotopic in $S^2 - \overline{f^{-1}(B)} \subset S^2 - A$ to $\widetilde{\gamma}_{\sigma(i)}$ and $\deg(f|\delta_i \colon \delta_i \to \widetilde{\gamma}_i) = 1$. Thus, $\widetilde{\mathcal{G}}$ provides a Levy cycle for $f \colon (S^2, A) \rto$. Moreover, by the second part of Corollary~\ref{corr: homotopy lifting for curves}, if $\mathcal{G}$ is a degenerate Levy cycle for $f \colon (S^2, B) \rto$, then $\widetilde{\mathcal{G}}$ is also a degenerate Levy cycle for $f \colon (S^2, A) \rto$. Similarly, if $\mathcal{G}$ is a multicurve, then $\widetilde{\mathcal{G}}$ is a multicurve as well, establishing item \eqref{it: realizability and levy cycles}.

     Suppose that $\mathcal{G} = \{\gamma_1, \gamma_2, \dots, \gamma_r\}$ is a Levy cycle for the Thurston map $f \colon (S^2, A) \rto$. If $\mathcal{G}$ were not a Levy cycle for $f \colon (S^2, B) \rto$, it would mean that some curve $\gamma_j \in \mathcal{G}$ is non-essential in $S^2 - B$. That implies that if $\delta \subset f^{-1}(\gamma_j)$ is a simple closed curve, then $\delta$ is non-essential in $S^2 - A$. However, there exists an essential simple closed curve $\widetilde{\gamma}_j \subset f^{-1}(\gamma_j)$ in $S^2 - A$, which leads to a contradiction. Finally, it is clear that if $\mathcal{G}$ is a degenerate Levy cycle for the Thurston map $f \colon (S^2, A) \rto$, then it is also degenerate for the Thurston map~$f \colon (S^2, B) \rto$.
\end{proof}

\begin{remark}
    Suppose we are in the setting of Proposition \ref{prop: trivial marked points}. Let $m \geq 1$ be the minimal integer such that every point of $C$ lands either to a point in $B$ within at most $m$ iterations, or to an essential singularity of $f$ after at most $m - 1$ iterations. Based on the ideas from the proof of item \eqref{it: realizability and trivial marked points}, and Propositions \ref{prop: small set} and \ref{prop: dependence}, it can be shown that if $|B| = 3$, then $\sigma_{f, A}^{\circ m}$ is constant, although $\sigma_{f, A}^{\circ n}$ is never constant for $1 \leq n \leq m - 1$.
\end{remark}

\subsection{Single marked fixed point}\label{subsec: extra fixed point}

Let $f \colon (S^2, B) \rto$ be a Thurston map with $|B| \geq 3$. If $c \in S^2 - B$ is a fixed point of the map $f$, then we can analyze the Thurston map $f \colon (S^2, A) \rto$, where $A := B \cup \{c\}$, using the information about the original Thurston map $f \colon (S^2, B) \rto$.

In this section, we assume that $f \colon (S^2, B) \rto$ is realized. This allows us to gain a deeper understanding of the behavior of the pullback map $\sigma_{f, A}$, which is defined on the Teichm\"uller space~$\T_{A}$.

According to Proposition \ref{prop: fixed point of sigma}, there exists a fixed point $\mu = [\theta_1] = [\theta_2] \in \T_B$ of the pullback map $\sigma_{f, B}$. We assume that the representatives $\theta_1, \theta_2 \colon S^2 \to \widehat{\C}$ are chosen such that $\theta_1$ and $\theta_2$ are isotopic rel.\ $B$ (in particular, $\theta_1|B = \theta_2|B$) and $g := \theta_1 \circ f \circ \theta_2^{-1} \colon (\widehat{\C}, P) \rto$ is a postsingularly finite holomorphic map, where $P := \theta_1(B) = \theta_2(B)$.


Now, we are ready to introduce the following objects:
\begin{itemize}
    \item $\Delta := t_{A, B}^{-1}(\mu)$, which is a $\sigma_{f, A}$-invariant subset of $\T_A$ according to Proposition \ref{prop: dependence};
    \item $\pi \colon \Delta \to \widehat{\C} - P$, defined by $\pi([\varphi]) = \varphi(c)$, where the representative $\varphi \colon S^2 \to \widehat{\C}$ is chosen so that $\varphi|B = \theta_1|B = \theta_2|B$;
    \item $\omega \colon \Delta \to \widehat{\C} - P$, defined as $\omega := \pi \circ \sigma_{f, A}|\Delta$;
    \item $W := \widehat{\C} - \overline{g^{-1}(P)}$, which is a domain of $\widehat{\C}$. 
\end{itemize}

Naturally, the objects introduced above depend on the map $f$, the marked set $B$, the fixed point~$c$, as well as the choice of the representatives $\theta_1$ and $\theta_2$ for the fixed point $\mu \in \T_B$ (which might not be unique if $f$ is a $(2, 2, 2, 2)$-map; see Remark \ref{rem: remark on 2,2,2,2-maps}). However, for simplicity, we exclude these dependencies from the notation.

\begin{proposition}\label{prop: extra marked point}
\begin{figure}[h]
        \begin{tikzcd}
            \Delta \arrow[rr, "\sigma_{f, A}|\Delta"] \arrow[dd, "\pi"] \arrow[dr, "\omega"] & & \Delta \arrow [dd, "\pi"]\\
            & W \arrow[dl, "g|W"] \arrow[dr, hook] & \\
            \widehat{\C} - P  & & \widehat{\C} - P            \end{tikzcd} 
        \caption{Fundamental diagram.}\label{fig: fund diag}
        \end{figure}
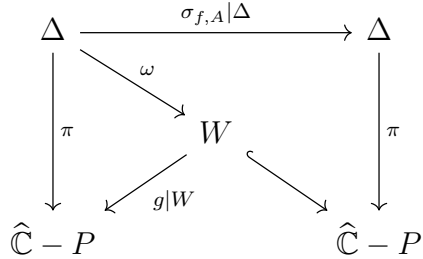
    The objects defined above satisfy the following properties:
    \begin{enumerate}
        \item \label{it: diag} diagram \eqref{fig: fund diag} commutes.  
        \item \label{it: disk} $\Delta$ is a properly and holomorphically embedded unit disk in $\T_A$;
        \item \label{it: coverings} $\pi \colon \Delta \to \widehat{\C} - P$, $g|W \colon W \to \widehat{\C} - P$, and $\omega\colon \Delta \to W$ are holomorphic covering maps;
        \item \label{it: non auto sigma map} $\sigma_{f, A}(\Delta) \subsetneq \Delta$ is open and dense in $\Delta$, and $\sigma_{f, A}|\Delta \colon \Delta \to \sigma_{f, A}(\Delta)$ is a holomorphic covering map.
    \end{enumerate}
\end{proposition}

\begin{proof}
    Let $\tau = [\varphi] \in \Delta$, where the representative $\varphi \colon S^2 \to \widehat{\C}$ is chosen such that $\varphi|B = \theta_1|B = \theta_2|B$. By Propositions \ref{prop: def of sigma map} and \ref{prop: dependence}, there exists an orientation-preserving homeomorphism $\psi \colon S^2 \to \widehat{\C}$ such that $g = \varphi \circ f \circ \psi^{-1}$ and $\psi|B = \theta_1|B = \theta_2|B$. Clearly, $g(\psi(c)) = \varphi(c)$. At the same time, $\pi(\tau) = \varphi(c)$, and $\omega(\tau) = \psi(c)$ because $\sigma_{f, A}(\tau) = [\psi]$. Thus, we have $\pi = \omega \circ g$. Using this, we conclude that diagram \eqref{fig: fund diag} commutes, establishing item \eqref{it: diag}.
    
    Item \eqref{it: disk} follows directly from Proposition \ref{prop: fibers}. The map $g|W \colon W \to \widehat{\C} - P$ is a covering because $S_g \subset P$. Similarly, the map $\pi \colon \Delta \to \widehat{\C} - P$ is a holomorphic covering map due to the discussion of Section \ref{subsec: teichmuller and moduli spaces}, and particularly, commutative diagram \eqref{fig: marked points forgetting}. Therefore, using the relation $\pi = \omega \circ g$ and \cite[Section 1.3, Excercise 16]{hatcher}, we deduce that $\omega \colon \Delta \to W$ is also a holomorphic covering map, thereby completing item \eqref{it: coverings}.
    
    Similarly, since $\omega = \pi \circ \sigma_{f, A}| \Delta$, the restriction $\sigma_{f, A}|\Delta \colon \Delta \to \sigma_{f, A}(\Delta)$ is a covering map, which is holomorphic because $\Delta$ is a complex submanifold of $\T_A$ and $\sigma_{f, A}$ is holomorphic on $\T_A$ by item~\eqref{it: hol} of Proposition \ref{prop: propeties of pullback map}. Note that by Great Picard's Theorem and the Riemann-Hurwitz formula \cite[Appendix A.3]{Hubbard_Book_1}, the set $\widehat{\C} - W$ contains at least one point different from the points of the set $P$. Therefore, $\sigma_{f, A}(\Delta) = \pi^{-1}(W)$ is different from $\Delta$. Since $W$ is open and dense in $\widehat{\C}$, it follows that $\sigma_{f, A}(\Delta)$ is open and dense in $\Delta$, establishing item~\eqref{it: non auto sigma map}.
\end{proof}

\subsection{Fixed points of pullback maps}\label{subsec: realized thurston maps with extra marked points}

The following proposition establishes the relationship between the fixed points of the pullback maps $\sigma_{f, A}$ and $\sigma_{f, B}$, where $f \colon (S^2, A) \rto$ is a Thurston map and $B \subset A$ is an $f$-pseudo-invariant subset containing $S_f$.

\begin{proposition}\label{prop: fixed points in fibers}
    Let $f \colon (S^2, A) \rto$ be a realized Thurston map, and let $B \subset A$ be an $f$-pseudo-invariant set with $|B| \geq 3$ and $S_f \subset B$. If $\mu \in \T_B$ is a fixed point of the pullback map $\sigma_{f, B}$, then the pullback map $\sigma_{f, A}$ has a unique fixed point in $t_{A, B}^{-1}(\mu)$. Furthermore, every fixed point of $\sigma_{f, A}$ belongs to the $t_{A, B}$-fiber of a fixed point of $\sigma_{f, B}$.
\end{proposition}

\begin{proof}
    First of all, if $\tau$ is a fixed point of $\sigma_{f, A}$, then by Proposition \ref{prop: dependence}, $t_{A, B}(\tau)$ is a fixed point of $\sigma_{f, B}$. Next, we will show that the pullback map  $\sigma_{f, A}$ has a unique fixed point in $t_{A, B}^{-1}(\mu)$, where $\mu$ is a fixed point of the pullback map $\sigma_{f, B}$.
    
    The answer is straightforward for Thurston maps that are not $(2, 2, 2, 2)$-maps due to item~\eqref{it: not 2,2,2,2 case} of Proposition \ref{prop: propeties of pullback map}, which implies that the associated pullback maps can have at most one fixed point. However, the case of $(2, 2, 2, 2)$-maps is non-trivial.

    First, we demonstrate that trivial marked points of the set $C := A - B$ can be ignored.

    \begin{claim1}
        Let $C = C_1 \sqcup C_2$, where $C_2$ consists of trivial marked points of the Thurston map $f\colon (S^2, A) \rto$. If the pullback map $\sigma_{f, B \cup C_1}$ has a unique fixed point in $t^{-1}_{B \cup C_1, B}(\mu)$, then the pullback map $\sigma_{f, A}$ also has a unique fixed point in $t^{-1}_{A, B}(\mu)$.
    \end{claim1}

    \begin{subproof}
        Let $\tau'$ be the unique fixed point of $\sigma_{f, B \cup C_1}$ such that $\tau' \in t^{-1}_{B \cup C_1, B}(\mu)$. According to item \eqref{it: realizability and trivial marked points} of Proposition \ref{prop: trivial marked points}, there exists a unique fixed point of $\sigma_{f, A}$ within $t_{A, B \cup C_1}^{-1}(\tau')$.

        Now, suppose that $\tau$ is any fixed point of $\sigma_{f, A}$ in $t^{-1}_{A, B}(\mu)$. By Proposition \ref{prop: dependence}, $t_{A, B \cup C_1}(\tau)$ is a fixed point of $\sigma_{f, B \cup C_1}$. Since $t_{A, B} = t_{B \cup C_1, C_1} \circ t_{A, B \cup C_1}$, it follows that $t_{A, B \cup C_1}(\tau) \in t_{B \cup C_1, B}^{-1}(\mu)$, and therefore, $t_{A, B \cup C_1}(\tau) = \tau'$. Thus, any fixed point of $\sigma_{f, A}$ in $t^{-1}_{A, B}(\mu)$ must belong to $t_{A, B \cup C_1}^{-1}(\tau')$, and the rest follows.
    \end{subproof}

    Next, we address the case when all points of the set $C = A - B = \{c_1, c_2, \dots, c_l\}$ are fixed under the map $f$.

    \begin{claim1}
        Suppose that every point $c \in C$ is fixed point of $f$. Then the pullback map $\sigma_{f, A}$ has a unique fixed point in $t_{A, B}^{-1}(\mu)$.
    \end{claim1}

    \begin{subproof}
        Define $B_i = B \cup \{c_1, c_2, \dots, c_i\}$ for each $i = 0, 1, \dots, l$. In particular, $B_0 = B$ and $B_l = A$. We prove by induction on $i$ with $0 \leq i \leq l$ that $\sigma_{f, B_i}$ has a unique fixed point in $t_{B_i, B}^{-1}(\mu)$. The base case is obvious, since $t_{B_0, B}^{-1}(\mu) = \{\mu\}$ and $\mu$ is a fixed point of $\sigma_{f, B_0}$.

        Now suppose that for some $j$ with $0 \leq j < l$, the pullback map $\sigma_{f, B_j}$ has a unique fixed point $\tau_j \in t_{B_j, B}^{-1}(\mu)$. Define $\Delta := t_{B_{j + 1}, B_j}^{-1}(\tau_j)$. According to Proposition \ref{prop: fibers}, $\Delta$ is a properly and holomorphically embedded unit disk in the Teichm\"uller space $\T_{B_{j + 1}}$. Furthermore, by Proposition \ref{prop: dependence}, $\Delta$ is $\sigma_{f, B_{j + 1}}$-invariant, and according to item \eqref{it: non auto sigma map} of Proposition \ref{prop: extra marked point}, the restriction $\sigma_{f, B_{j + 1}}|\Delta \colon \Delta \to \Delta$ is holomorphic but not an automorphism of $\Delta$. 
        
        However, the $\sigma_{f, B_{j + 1}}$-orbit of any point in $\Delta$ forms a pre-compact set within $\T_{B_{j + 1}}$, and thus within $\Delta$ as well. This follows easily, for example, from the facts that $\sigma_{f, B_{j + 1}}$ has a fixed point in $\T_{B_{j + 1}}$ due to Proposition \ref{prop: fixed point of sigma}, and that $\sigma_{f, B_{j + 1}}$ is 1-Lipschitz on $\T_{B_{j + 1}}$, as noted in item \eqref{it: general} of Proposition \ref{prop: propeties of pullback map}. As a result, by the Denjoy-Wolff theorem \cite[Theorem~3.2.1]{Abate}, the pullback map $\sigma_{f, B_{j + 1}}$ has a unique fixed point in $\Delta$.

        Now, suppose that $\tau$ is any fixed point of the pullback map $\sigma_{f, B_{j + 1}}$ in $t^{-1}_{B_{j + 1}, B}(\mu)$. By Proposition \ref{prop: dependence}, it is straightforward to see that $t_{B_{j + 1}, B_j}(\tau)$ is a fixed point of $\sigma_{f, B_j}$. Moreover, since $t_{B_{j + 1}, B} = t_{B_j, B} \circ t_{B_{j + 1}, B_{j}}$, we have $t_{B_{j + 1}, B_j}(\tau) \in t_{B_j, B}^{-1}(\mu)$. As $\tau_j$ is unique fixed point of $\sigma_{f, B_j}$ in $t_{B_j, B}^{-1}(\mu)$, it follows $t_{B_{j + 1}, B_j}(\tau) = \tau_j$, or equivalently, $\tau \in t_{B_{j + 1}, B_j}^{-1}(\mu) = \Delta$. Since we have already shown that $\sigma_{f, B_{j + 1}}$ has a unique fixed point in $\Delta$, the uniqueness of the fixed point of $\sigma_{f, B_{j + 1}}$ within $t_{B_{j + 1}, B}^{-1}(\mu)$ follows.
    \end{subproof}

    Now, consider the general case. By Claim 1, we may assume that the set $C$ does not contain any trivial marked points of the Thurston map $f \colon (S^2, A) \rto$. Therefore, every point $c \in C$ is periodic under $f$, and we can choose $m \geq 1$ such that every $c \in C$ is fixed under $f^{\circ m}$. Applying Claim 2 to the Thurston map $f^{\circ m} \colon (S^2, A) \rto$, we find that there is a unique fixed point $\tau$ of $\sigma_{f^{\circ m}, A}$ in $t_{A, B}^{-1}(\mu)$. By Proposition \ref{prop: functoriality}, we know that $\sigma_{f^{\circ m}, A} = \sigma_{f, A}^{\circ m}$, and in particular, $\tau$ is a periodic point of $\sigma_{f, A}$.
    
    If $\tau$ is a fixed point of $\sigma_{f, A}$, we are done, because any other fixed point of $\sigma_{f, A}$ in $t_{A, B}^{-1}(\mu)$ would lead to a fixed point of $\sigma_{f^{\circ m}, A} = (\sigma_{f, A})^{\circ m}$, contradicting Claim 2. Now, suppose that $\tau$ is not a fixed point of $\sigma_{f, A}$. Then $\tau$ and $\sigma_{f, A}(\tau)$ would be distinct fixed points of $\sigma_{f^{\circ m}, A}$. However, both $\tau$ and $\sigma_{f, A}(\tau)$ lie in $t_{A, B}^{-1}(\mu)$, as $t_{A, B}^{-1}(\mu)$ is $\sigma_{f, A}$-invariant by Proposition \ref{prop: dependence}. This leads to a contradiction with Claim 2.
\end{proof}

\begin{corollary}\label{corr: corr from intro}
    Let $f \colon (S^2, A) \rto$ be a realized Thurston map. Suppose that $B \subset A$ is an $f$-pseudo-invariant set containing $S_f$, and there exist two orientation-preserving homeomorphisms $\varphi, \psi \colon S^2 \to \widehat{\C}$ such that $\varphi$ and $\psi$ are isotopic rel.\ $B$ and $g = \varphi \circ f \circ  \psi^{-1}\colon (\widehat{\C}, \varphi(B))\rto$ is a holomorphic postsingularly finite map. Then, there exist two orientation-preserving homeomorphisms $\widetilde{\varphi}, \widetilde{\psi} \colon S^2 \to \widehat{\C}$ such that 
    \begin{enumerate}
        \item\label{it: 1} $\widetilde{\varphi}$ and $\widetilde{\psi}$ are isotopic rel.\ $A$;
        \item\label{it: 2} $\varphi$ and $\widetilde{\varphi}$ are isotopic rel.\ $B$, and $\psi$ and $\widetilde{\psi}$ are isotopic rel.\ $B$;
        \item\label{it: 3} $g = \widetilde{\varphi} \circ f \circ \widetilde{\psi}^{-1} \colon (\widehat{\C}, \widetilde{\varphi}(A)) \rto$.
    \end{enumerate}
    Moreover, if $|B| \geq 3$, the orientation-preserving homeomorphisms $\widetilde{\varphi}$ and $\widetilde{\psi}$ that satisfy properties \eqref{it: 1}-\eqref{it: 3} are unique up isotopy rel.\ $A$.
\end{corollary}

\begin{proof}
    This result is follows directly from Proposition \ref{prop: fixed points in fibers}, along with Proposition \ref{prop: dependence} and \cite[Proposition~2.3]{farb_margalit}.
\end{proof}

\subsection{Passing to an iterate}\label{subsec: passing to an iterate}

The following proposition establishes a connection between certain properties of a marked Thurston map and of those of its iterates:

\begin{proposition}\label{prop: passing to iterate}
    Let $f \colon (S^2, A) \rto$ be a Thurston map and $m \geq 1$. Then:
    \begin{enumerate}
        \item \label{it: realizability and iterate} $f \colon (S^2, A) \rto$ is realized if and only if $f^{\circ m} \colon (S^2, A) \rto$ is realized;

        \item \label{it: descending from iterate} if $f^{\circ m} \colon (S^2, A) \rto$ has a Levy cycle (or a degenerate Levy cycle) $\mathcal{G}$ of period $r$, then $f \colon (S^2, A) \rto$ has a Levy cycle (or a degenerate Levy cycle, respectively) $\mathcal{F}$ of period~$mr$ such that $\mathcal{G} \subset \mathcal{F}$;
        
        \item \label{it: descedning from iterate for multicurves} $f^{\circ m} \colon (S^2, A) \rto$ has a Levy cycle (or a degenerate Levy cycle) $\mathcal{G}$ and there exist $\gamma \in \mathcal{G}$ and $\tau \in \T_A$ such that $l_{\gamma}(\tau) < \ell^* e^{-(mr - 1)d_0}$, where $d_0 := d_T(\tau, \sigma_{f, A}(\tau))$, then $f \colon (S^2, A) \rto$ has a Levy multicurve (or a degenerate Levy multicurve, respectively)~$\mathcal{F}$ of period at most $mr$;

        \item \label{it: passing to iterate} if $f \colon (S^2, A) \rto$ has a Levy cycle (or a degenerate Levy cycle, Levy multicurve, or degenerate Levy multicurve) $\mathcal{G}$ of period $r$, then $f^{\circ m} \colon (S^2, A) \rto$ has a Levy cycle (or degenerate Levy cycle, Levy multicurve, or degenerate Levy multicurve, respectively) $\mathcal{F}$ of period $r / \mathrm{gcd}(m, r)$ such that~$\mathcal{F} \subset \mathcal{G}$.
    \end{enumerate}
\end{proposition}

\begin{proof}
    For item \eqref{it: realizability and iterate}, we can assume that $|A| \geq 3$ due to item \eqref{it: three points} of Proposition \ref{prop: small set}.
    Suppose the Thurston map $f\colon (S^2, A) \rto$ is realized. According to Proposition~\ref{prop: fixed point of sigma}, the pullback map $\sigma_{f, A}$ has a fixed point $\tau \in \T_A$. By Proposition \ref{prop: functoriality}, $\tau$ is also a fixed point of $\sigma_{f^{\circ m}, A}$, since $\sigma_{f^{\circ m}, A} = (\sigma_{f, A})^{\circ m}$. From Proposition \ref{prop: fixed point of sigma}, we conclude that $f^{\circ m} \colon (S^2, A) \rto$ is also realized. 
    
    Now, suppose that $f^{\circ m} \colon (S^2, A) \rto$ is realized. As previously discussed, this implies the existence of a periodic point of period $m$ for the pullback map $\sigma_{f, A}$. If $f$ is not a $(2, 2, 2, 2)$-map, then by item \eqref{it: not 2,2,2,2 case} of Proposition \ref{prop: propeties of pullback map}, every periodic point of $\sigma_{f, A}$ must be fixed, and the desired result follows from Proposition~\ref{prop: fixed point of sigma}. 
    
    Suppose that $f$ is a $(2, 2, 2, 2)$-map. As discussed in Section~\ref{subsec: teichmuller and moduli spaces}, the Teichm\"uller space $\T_{P_f}$ is biholomorphic to the unit disk. Since $\sigma_{f, P_f}$ has a periodic point by Proposition~\ref{prop: dependence}, then Schwarz's lemma and the classification of automorphisms of $\D$ imply that $\sigma_{f, P_f}$ must have a fixed point $\mu$. Clearly, $\mu$ is also a fixed point of $\sigma_{f^{\circ m}, P_f}$. Therefore, according to Proposition~\ref{prop: fixed points in fibers}, $\sigma_{f^{\circ m}, A}$ has a unique fixed point $\tau$ in $t_{A, P_f}^{-1}(\mu)$. Note that both $\tau$ and $\sigma_{f, A}(\tau)$ are fixed points of $\sigma_{f^{\circ m}, A}$. Moreover, $\sigma_{f, A}(\tau) \in t_{A, P_f}^{-1}(\mu)$, since $t_{A, P_f}^{-1}(\mu)$ is $\sigma_{f, A}$-invariant according to Proposition~\ref{prop: dependence}.
    Therefore, $\tau$ and $\sigma_{f, A}(\tau)$ must coincide, and item~\eqref{it: realizability and iterate} follows from Proposition~\ref{prop: fixed point of sigma}.
    

    Suppose $f^{\circ m} \colon (S^2, A) \rto$ has a Levy cycle $\mathcal{G} = \{\gamma_1, \gamma_2, \dots, \gamma_r\}$. Let $\sigma$ denote the corresponding cyclic permutation of the set $\{1, 2, \dots, r\}$, and for each $i = 1, 2, \dots, r$, let $\widetilde{\gamma}_i \subset f^{-m}(\gamma_i)$ be a simple closed curve such that $\widetilde{\gamma}_i$ is homotopic in $S^2 - A$ to $\gamma_{\sigma(i)}$ and $\deg(f^{\circ m}|\widetilde{\gamma}_i \colon \widetilde{\gamma}_i \to \gamma_i) = 1$. Denote $\gamma_{i, k} := f^{\circ (m - k)}(\widetilde{\gamma}_i)$ for each $i = 0, 1, \dots m - 1$. In particular, $\gamma_{i, 0} = \gamma_i$. It is clear that each $\gamma_{i, k}$ must be an essential curve in $S^2 - A$, as $\widetilde{\gamma}_i$ would not be essential otherwise. It is straightforward to verify that $$\mathcal{F} := \{\gamma_{i, k}: i = 1, 2, \dots, r \text{ and } k = 0, 1, \dots, m-1\}$$ provides a Levy cycle for the map $f \colon (S^2, A) \rto$. Similarly, if $\mathcal{G}$ is a degenerate Levy cycle for $f^{\circ m}\colon (S^2, A) \rto$, then $\mathcal{F}$ is a degenerate Levy cycle for $f \colon (S^2, A) \rto$, and this establishes item \eqref{it: descending from iterate}.

    Suppose that additionally there exist $\gamma \in \mathcal{G}$ and $\tau \in \T_A$ such that $l_{\gamma}(\tau) < \ell^* e^{-(mr - 1)d_0}$. Let us choose a minimal Levy cycle $\mathcal{F}'$ such that $\gamma \in \mathcal{F}'$ and $\mathcal{F}' \subset \mathcal{F}$. With this, item \eqref{it: descedning from iterate for multicurves} follows from Proposition \ref{prop: levy cycles to levy multicurves} applied to the minimal Levy cycle $\mathcal{F}'$.

    Finally, item \eqref{it: passing to iterate} can be easily derived from the definition of a Levy cycle (degenerate Levy cycle, Levy multicurve, degenerate Levy multicurve, respectively); see Section \ref{subsec: obstructions}.
\end{proof}

\subsection{Realizability question}\label{subsec: thurston maps with extra marked points}

In this section, we investigate the properties of a Thurston map $f \colon (S^2, A) \rto$ under the assumption that the Thurston map $f \colon (S^2, B) \rto$ is realized, where $B \subset A$ is some $f$-pseudo-invariant subset  containing~$S_f$. In particular, we establish Main Theorem \ref{mainthm A}. We begin by proving a result that demonstrates that Levy cycles for the map $f \colon (S^2, A) \rto$ must exhibit particularly nice properties.

\begin{proposition}\label{prop: forgetting points and levy cycles}
    Let $f \colon (S^2, A) \rto$ be a Thurston map, and let $B \subset A$ be an $f$-pseudo-invariant subset containing $S_f$. Suppose that the Thurston map $f \colon (S^2, B) \rto$ is realized, and the Thurston map $f \colon (S^2, A) \rto$ has a Levy cycle $\mathcal{G}$. Then $\mathcal{G}$ is degenerate, and every $\gamma \in \mathcal{G}$ is non-essential in $S^2 - B$.
\end{proposition}

\begin{proof}
    Let $\mathcal{G} = \{\gamma_1, \gamma_2, \dots, \gamma_r\}$, and let $\sigma$ denote the corresponding cyclic permutation of the set $\{1, 2, \dots, r\}$. Then, for each $i = 1, 2, \dots, r$, there exists a simple closed curve $\widetilde{\gamma}_i \subset f^{-1}(\gamma_i)$ such that $\widetilde{\gamma}_i$ is homotopic in $S^2 - B$ to $\gamma_{\sigma(i)}$ and $\deg(f|\widetilde{\gamma}_i \colon \widetilde{\gamma}_i \to \gamma_i) = 1$.

    Suppose that some $\gamma_j$ is essential in $S^2 - B$. Then $\widetilde{\gamma}_{\sigma^{-1}(j)}$ is also essential in $S^2 - B$, since it is homotopic in $S^2 - A$ to $\gamma_j$. This implies that $\gamma_{\sigma^{-1}(j)}$ is essential in $S^2 - B$. Repeatedly applying this argument, we show that each $\gamma_i \in \mathcal{G}$ is essential in $S^2 - B$, and therefore, $\mathcal{G}$ is a Levy cycle for the Thurston map $f\colon (S^2, B) \rto$. However, this leads to a contradiction with Proposition~\ref{prop: levy cycles are obstructions}.

    Now, let us demonstrate that $\mathcal{G}$ forms a degenerate Levy cycle for $f \colon (S^2, A) \rto$. First, we suppose that $|B| \geq 3$. For each $i=1, 2, \dots, r$, since $\gamma_i$ is non-essential in $S^2 - B$, we can select the unique connected component $D_i$ of $S^2 - \gamma_i$ that contains at most one point of $B$. Similarly, we introduce $\widetilde{D}_i$ as the unique connected component of $S^2 - \widetilde{\gamma}_i$ that contains at most one point of $B$.

    Suppose that some $D_i$ contains a single point $q_i$ from the set $B$. Since $S_f \subset B$, by the classical theory of covering maps (see, for instance, \cite[Propositions 2.11 and 2.12]{our_approx}), there exists a unique connected component $U_i$ of $S^2 - \widetilde{\gamma}_i$ that contains a single point $p_i$ of $\overline{f^{-1}(B)} \supset B$. Moreover, $f(p_i) = q_i$, and $f|U_i - \{p_i\} \colon U_i - \{p_i\} \to D_i - \{q_i\}$ is a covering map of degree $\deg(f|\widetilde{\gamma}_i \colon \widetilde{\gamma}_i \to \gamma_i) = \deg(f, p_i)$. Clearly, $U_i$ must coincide with $\widetilde{D}_i$. Since $\deg(f|\widetilde{\gamma}_i \colon \widetilde{\gamma}_i \to \gamma_i) = 1$, we also have $\deg(f|\widetilde{D}_i \colon \widetilde{D}_i \to D_i) = 1$. Similarly, when $D_i \cap B = \emptyset$, we can also conclude that $f(\widetilde{D}_i) = D_i$ and $\deg(f|\widetilde{D}_i \colon \widetilde{D}_i \to D_i) = 1$.

    Finally, it is easy to verify that $D_{\sigma(i)}$ and $\widetilde{D}_i$ are homotopic in $S^2 - A$. Indeed, there exists an isotopy $(\varphi_t)_{t \in \I}$ in $S^2$ rel.\ $A$ with $\varphi_0 = \id_{S^2}$ and $\varphi_1(\partial D_{\sigma(i)}) = \varphi_1(\gamma_{\sigma(i)}) = \widetilde{\gamma}_i = \partial \widetilde{D}_i$ (\cite[Theorem A.3]{BuserGeometry}; see also \cite[Sections 1.2.5 and 1.2.6]{farb_margalit}). Since both $D_{\sigma(i)}$ and $\widetilde{D}_i$ contain at most one point of $B$ and $|B| \geq 3$, we have $\varphi_1(D_{\sigma(i)}) = \widetilde{D}_i$. Thus, $D_{\sigma(i)}$ and $\widetilde{D}_i$ are homotopic in $S^2 - A$, implying that $\mathcal{G}$ is a degenerate Levy cycle for the Thurston map $f \colon (S^2, A) \rto$.
    
    In the case when the set $B$ contains fewer then three points, we can use Proposition \ref{prop: small set} to prove that $\mathcal{G}$ is a degenerate Levy cycle for the Thurston map $f \colon (S^2, A) \rto$. In particular, open Jordan regions $D_i$ and $\widetilde{D}_i$ can be defined as the unique connected components of $S^2 - \gamma_i$ and $S^2 - \widetilde{\gamma}_i$, respectively, that contain no points of the set~$B$.
\end{proof}

A point $z \in \widehat{\C}$ is called a \textit{fixed point} of a holomorphic map $g \colon U \to V$, where $U$ and $V$ are domains of $\widehat{\C}$, if either $z \in U$ and $g(z) = z$, or $z$ is an isolated removable singularity of $g$ and, after extending $g$ holomorphically to a neighborhood $z$, we have $g(z) = z$. The notion of a repelling fixed point can be generalized in a similar~way.

Before proceeding, we present a result regarding the dynamics of holomorphic self-maps of the unit disk that satisfy certain additional conditions \cite[Theorem 3.10]{small_postsingular_set}.

\begin{theorem}\label{thm: iteration on unit disk}
    Let $h \colon \D \to \D$ be a holomorphic map, and $\pi \colon \D \to V$ and $g \colon U \to V$ are holomorphic covering maps, where $U \subset V$ is a domain of $\widehat{\C}$ and $V = \widehat{\C} - P$ with $3 \leq |P| < \infty$. Suppose that $\pi(h(\D)) \subset U$ and $\pi = g \circ \pi \circ h$, i.e., the following diagram commutes:
    \begin{center}  
        \begin{tikzcd}
            \D \arrow[r, "h"] \arrow[d, "\pi"] & h(\D) \arrow [d, "\pi"]\\
            V & U \arrow[l, "g"]
        \end{tikzcd}
    \end{center}
    If the map $g \colon U \to V$ is non-injective, then exactly one of the following two possibilities is~satisfied:
    \begin{enumerate}
        \item for every $z \in \D$, the $h$-orbit of $z$ converges to the unique fixed point of $h$, or

        \item the sequence $(\pi(h^{\circ n}(z)))$ converges to the same repelling fixed point $x \in P$ of the map~$g$, regardless of the choice of $z \in \D$.
    \end{enumerate}
\end{theorem}

Now we are ready to prove Main Theorem \ref{mainthm A} from the introduction.

\begin{proof}[Proof of Main Theorem \ref{mainthm A}]
    As stated in Proposition \ref{prop: levy cycles are obstructions}, if the Thurston map $f \colon (S^2, A) \rto$ has a Levy cycle, then $f \colon (S^2, A) \rto$ is obstructed, establishing one of the directions of the desired result. Therefore, from this point, we assume that $f \colon (S^2, A) \rto$ is obstructed.
    According to Proposition \ref{prop: forgetting points and levy cycles}, it remains to prove that $f \colon (S^2, A) \rto$ has a Levy multicurve.
    
    At the same time, items~\eqref{it: realizability and trivial marked points} and~\eqref{it: realizability and levy cycles} of Proposition~\ref{prop: trivial marked points} allow us to assume that the set $C : = A - B$ contains no trivial marked points, meaning that every point $c \in C$ is periodic. We will focus on the case when every point $c \in C$ is fixed under $f$. We aim to show that for any such obstructed Thurston map $f \colon (S^2, A) \rto$ (satisfying the assumptions of this theorem) and for every $\varepsilon > 0$, there exists a Levy cycle~$\mathcal{G}$, a curve $\gamma \in \mathcal{G}$, and a point $\tau$ in the Teichm\"uller space $\T_A$ such that $l_{\gamma}(\tau) < \varepsilon$. Note that all marked points in $C$ can be made fixed by passing to an appropriate iterate of the original Thurston map. Therefore, items \eqref{it: realizability and iterate} and~\eqref{it: descedning from iterate for multicurves} of Proposition \ref{prop: passing to iterate} imply that the scenario described above is sufficient to conclude the desired result in the general case. Since under these assumptions, every point $c \in C$ is fixed under $f$, we can suppose without loss of generality that there is no point $c \in C$ such that the Thurston map $f \colon (S^2, B \cup \{c\}) \rto$ is realized. This assumption and item \eqref{it: three points} of Proposition \ref{prop: small set} also ensure that $|B| \geq 3$.

    According to Proposition \ref{prop: fixed point of sigma}, there exists a fixed point $\mu = [\varphi_1]_B = [\psi_1]_B \in \T_{B}$ of the pullback map $\sigma_{f, B}$. We assume that the representatives $\varphi_1, \psi_1 \colon S^2 \to \widehat{\C}$ are chosen such that $\varphi_1$ and $\psi_1$ are isotopic rel.\ $B$ (in particular, $\varphi_1|B = \psi_1|B$) and $g := \varphi_1 \circ f \circ \psi_1^{-1} \colon (\widehat{\C}, P) \rto$ is a postsingularly finite holomorphic map, where $P := \varphi_1(B) = \psi_1(B)$.

    Now, let $\tau_1 = [\varphi_1]_{A} \in \T_{A}$ and define $\tau_n := [\varphi_n]_{A} = \sigma_{f, A}^{\circ (n - 1)}(\tau_1)$ for all $n \in \N$, where the representatives $\varphi_n \colon S^2 \to \widehat{\C}$ are chosen such that $\varphi_n|B = \varphi_1|B$ and the map $\varphi_n \circ f \circ \varphi_{n + 1}^{-1} \colon \widehat{\C} \dto \widehat{\C}$ is holomorphic. Proposition \ref{prop: dependence} ensures that such choice of representatives is possible and implies that $\varphi_n \circ f \circ \varphi_{n + 1}^{-1} = g$ for all $n \in \N$. Let $C = \{c_1, c_2, \dots, c_l\}$, and define $x_{n, i} := \varphi_n(c_i)$ for each $n \in \N$ and $i = 1, 2, \dots, l$. 

    \begin{claim}
        For each $i=1,2,\dots, l$, the sequence $(x_{n, i})$ converges to a limit $x_i \in P$, which is a repelling fixed point of the map $g$.
    \end{claim}

    \begin{subproof}
        Let $A_i := B \cup \{c_i\}$ for each $i = 1, 2, \dots, l$. Since $f(c_i) = c_i$, we can study the Thurston map $f \colon (S^2, A_i) \rto$. Specifically, we consider the $\sigma_{f, A_i}$-orbit of the point $[\varphi_1]_{A_i}$ in the Teichm\"uller space $\T_{A_i}$. By Proposition \ref{prop: dependence}, this orbit coincides with the sequence $([\varphi_n]_{A_i})$.

        Define $\Delta_i := t_{A_i, B}^{-1}(\mu)$, which, according to Proposition \ref{prop: dependence}, is a $\sigma_{f, A_i}$-invariant subset of $\T_{A_i}$. Also, let $\pi_i \colon \Delta_i \to \widehat{\C} - P$ be defined by $\pi_i([\varphi]_{A_i}) = \varphi(c_i)$, where the representative $\varphi \colon S^2 \to \widehat{\C}$ is chosen so that $\varphi|B = \varphi_1|B$, and let $\omega_i = \pi_i \circ \sigma_{f, A_i}|\Delta_i$. Recall that these objects were discussed in Section \ref{subsec: extra fixed point}. Additionally, we notice that $[\varphi_n]_{A_i} \in \Delta_i$ for all $n\in \N$ and $i = 1, 2, \dots, l$.

        According to item \eqref{it: disk} of Proposition \ref{prop: extra marked point}, $\Delta_i$ is biholomorphic to the unit disk, and by item~\eqref{it: diag} of the same proposition, we have $\pi_i(\sigma_{f, A_i}(\Delta_i)) = W$, where $W = \widehat{\C} - \overline{g^{-1}(P)}$, and $\pi_i = g \circ \pi_i \circ \sigma_{f, A_i}$. Since $g$ is not injective, items \eqref{it: coverings} and \eqref{it: non auto sigma map} of Proposition \ref{prop: extra marked point} allow us to apply Theorem~\ref{thm: iteration on unit disk} to the map $\sigma_{f, A_i}|\Delta_i$. Since the Thurston map $f \colon (S^2, A_i) \rto$ is obstructed, $\sigma_{f, A_i}|\Delta_i$ has no fixed points, as stated in Proposition \ref{prop: fixed point of sigma}. Therefore, by Theorem~\ref{thm: iteration on unit disk}, the sequence $x_{n, i} = \pi_i([\varphi_n]_{A_i}) = \pi_i(\sigma_{f, A_i}^{\circ(n - 1)}([\varphi_1]_{A_i}))$ converges to a point $x_i \in P$, which is a repelling fixed point of the map~$g$.
    \end{subproof}

    Let $x \in P$ be a point such that there exists a sequence $(x_{n, j})$ converging to $x$ as $n$ tends to~$\infty$. According to the previous claim, such a point $x$ always exists and must be a repelling fixed point of the map $g$. Given this, and the fact that every sequence $(x_{n, i})$ converges as $n \to \infty$, we can find an annulus $U \subset \widehat{\C}$ such that, for all sufficiently large $n$, the following conditions are satisfied:
    \begin{itemize}
        \item one of the connected components of $\widehat{\C} - U$ contains points $x$ and $x_n^{j}$, while the other connected component contains all points of $P - \{x\}$;
        \item $g^{\circ (|A| - 3)}$ is defined and injective on $U$.
    \end{itemize}
    It is clear that we can take the modulus of the annulus $U$ as large as desired, assuming that we consider $n$ large enough. Therefore, we can apply Proposition \ref{prop: finding a levy cycle} to derive the existence of a Levy cycle $\mathcal{G}$ for the Thurston map $f \colon (S^2, A) \rto$. Moreover, Proposition \ref{prop: finding a levy cycle} also shows that, by taking $\mod(U)$ sufficiently large, we can ensure that for some $n$ and any $\gamma \in \mathcal{G}$, $l_{\gamma}(\tau_n)$ is as small as we want. This completes the proof.
\end{proof}

\begin{example}
    Let $f \colon (S^2, A) \rto$ be a Thurston map, and assume that $B \subset A$ is an $f$-pseudo-invariant subset containing $S_f$.

    Now, suppose $|A| = 4$ and $|B| \leq 3$. In this case, by Proposition \ref{prop: small set}, the Thurston map $f \colon (S^2, B) \rto$ is clearly realized. Consequently, Main Theorem \ref{mainthm A} directly implies that $f \colon (S^2, A) \rto$ is realized if and only if it has no degenerate Levy cycles. Since any multicurve on the four-punctured sphere $S^2 - A$ consists of just one simple closed curve, $f \colon (S^2, A) \rto$ is obstructed if and only if it has a degenerate Levy fixed curve, confirming the result \cite[Main Theorem A]{small_postsingular_set} in this particular setting.
\end{example}

\begin{remark}
    One of the key tools of the proof of Theorem \ref{thm: iteration on unit disk} (see \cite[Theorem 3.10]{small_postsingular_set}) is the use of the hyperbolic metric on the unit disk. In the proof of Main Theorem \ref{mainthm A}, Theorem~3.10 was applied to the restrictions $\sigma_{f, A_i}|\Delta_i \colon \Delta_i \to \Delta_i$, where $\Delta_i = t_{A_i, B}^{-1}(\mu)$ is a properly and holomorphically embedded unit disk in the Teichm\"uller space $\T_{A_i}$, where $|A_i| \geq 4$ and $|A_i| - |A| = 1$. However, it is worth noting, that the hyperbolic distance on $\Delta_i$ does not coincide with the metric inherited from the Teichm\"uller metric $d_T$ on $\T_{A_i}$, except when $|A_i| = 4$ and $\Delta_i = \T_{A_i}$, due to \cite[Theorem 1]{non_isometric_disks}. The same observation applies to properly and holomorphically embedded unit disk $\Delta = t_{B_{j + 1}, B_j}^{-1}(\tau_j)$ in $\T_{B_{j + 1}}$ and the map $\sigma_{f, B_{j + 1}}|\Delta \colon \Delta \to \Delta$ from the proof of Proposition~\ref{prop: fixed points in fibers}.
\end{remark}

\bibliographystyle{alpha}
\bibliography{lib.bib}

\begin{thebibliography}{BEKP09}

\bibitem[Aba23]{Abate}
Marco Abate.
\newblock {\em Holomorphic dynamics on hyperbolic {R}iemann surfaces}, volume~89 of {\em De Gruyter Studies in Mathematics}.
\newblock De Gruyter, Berlin, [2023] \copyright 2023.

\bibitem[ABF21]{Astorg_Benini_Fagella}
Matthieu Astorg, Anna~Miriam Benini, and N\'uria Fagella.
\newblock Bifurcation loci of families of finite type meromorphic maps.
\newblock {\em Preprint arXiv:2107.02663}, 2021.

\bibitem[Ast]{Astorg}
Matthieu Astorg.
\newblock Bifurcations and wandering domains in holomorphic dynamics.
\newblock \textit{Habilitation à Diriger des Recherches}, Université d'Orléans, 2024. Available at https://www.idpoisson.fr/astorg/papiers/hdrastorg.pdf.

\bibitem[BCT14]{Buff}
Xavier Buff, Guizhen Cui, and Lei Tan.
\newblock Teichm\"{u}ller spaces and holomorphic dynamics.
\newblock In {\em Handbook of {T}eichm\"{u}ller theory. {V}ol. {IV}}, volume~19 of {\em IRMA Lect. Math. Theor. Phys.}, pages 717--756. Eur. Math. Soc., Z\"{u}rich, 2014.

\bibitem[BD17]{overview_bartholdi_dudko}
Laurent Bartholdi and Dzmitry Dudko.
\newblock Algorithmic aspects of branched coverings.
\newblock {\em Ann. Fac. Sci. Toulouse Math. (6)}, 26(5):1219--1296, 2017.

\bibitem[BD21]{Bartholdi_Dudko}
Laurent Bartholdi and Dzmitry Dudko.
\newblock Algorithmic aspects of branched coverings {III}/{V}. {E}rasing maps, orbispaces, and the {B}irman exact sequence.
\newblock {\em Groups Geom. Dyn.}, 15(4):1197--1265, 2021.

\bibitem[BDP24]{correpondences}
Laurent Bartholdi, Dzmitry Dudko, and Kevin~M. Pilgrim.
\newblock Correspondences on {R}iemann surfaces and non-uniform hyperbolicity.
\newblock {\em Preprint arXiv: 2407.15548}, 2024.

\bibitem[BEKP09]{pullback_map}
Xavier Buff, Adam Epstein, Sarah Koch, and Kevin Pilgrim.
\newblock On {T}hurston's pullback map.
\newblock In {\em Complex dynamics}, pages 561--583. A K Peters, Wellesley, MA, 2009.

\bibitem[BF14]{branner_fagella_2014}
B.~Branner and N.~Fagella.
\newblock {\em Quasiconformal surgery in holomorphic dynamics}, volume 141 of {\em Cambridge Studies in Advanced Mathematics}.
\newblock Cambridge University Press, Cambridge, 2014.
\newblock With contributions by Xavier Buff, Shaun Bullett, Adam L. Epstein, Peter Ha\"{\i}ssinsky, Christian Henriksen, Carsten L. Petersen, Kevin M. Pilgrim, Tan Lei and Michael Yampolsky.

\bibitem[Bis15a]{qc_foldings}
Christopher~J. Bishop.
\newblock Constructing entire functions by quasiconformal folding.
\newblock {\em Acta Math.}, 214(1):1--60, 2015.

\bibitem[Bis15b]{order_conjecture}
Christopher~J. Bishop.
\newblock The order conjecture fails in {$\mathcal S$}.
\newblock {\em J. Anal. Math.}, 127:283--302, 2015.

\bibitem[Bis17]{models_for_class_S}
Christopher~J. Bishop.
\newblock Models for the {S}peiser class.
\newblock {\em Proc. Lond. Math. Soc. (3)}, 114(5):765--797, 2017.

\bibitem[Bus10]{BuserGeometry}
P.~Buser.
\newblock {\em Geometry and spectra of compact {R}iemann surfaces}.
\newblock Birkh\"auser, Boston, MA, 2010.

\bibitem[Con95]{conway}
John~B. Conway.
\newblock {\em Functions of one complex variable. {II}}, volume 159 of {\em Graduate Texts in Mathematics}.
\newblock Springer-Verlag, New York, 1995.

\bibitem[DH93]{DH_Th_char}
A.~Douady and J.H. Hubbard.
\newblock A proof of {T}hurston's topological characterization of rational functions.
\newblock {\em Acta Mathematica}, 171(2):263--297, 1993.

\bibitem[Eps93]{Adam_Thesis}
A.L. Epstein.
\newblock {\em Towers of finite type complex analytic maps}.
\newblock PhD thesis, The City University of New York, 1993.
\newblock Available at http://pcwww.liv.ac.uk/\textasciitilde lrempe/adam/thesis.pdf.

\bibitem[ERG15]{Rempe_Epstein}
Adam Epstein and Lasse Rempe-Gillen.
\newblock On invariance of order and the area property for finite-type entire functions.
\newblock {\em Ann. Acad. Sci. Fenn. Math.}, 40(2):573--599, 2015.

\bibitem[FM12]{farb_margalit}
B.~Farb and D.~Margalit.
\newblock {\em A Primer on Mapping Class Groups}.
\newblock Princeton Math. Ser. 49. Princeton Univ. Press, Princeton, NJ, 2012.

\bibitem[For91]{Forster}
Otto Forster.
\newblock {\em Lectures on {R}iemann surfaces}, volume~81 of {\em Graduate Texts in Mathematics}.
\newblock Springer-Verlag, New York, 1991.
\newblock Translated from the 1977 German original by Bruce Gilligan, Reprint of the 1981 English translation.

\bibitem[Hat02]{hatcher}
Allen Hatcher.
\newblock {\em Algebraic topology}.
\newblock Cambridge University Press, Cambridge, 2002.

\bibitem[HP22]{critically_fixed}
M.~Hlushchanka and N.~Prochorov.
\newblock Critically fixed {T}hurston maps: classification, recognition, and twisting.
\newblock {\em Preprint arXiv:2212.14759}, 2022.

\bibitem[HSS09]{HSS}
John Hubbard, Dierk Schleicher, and Mitsuhiro Shishikura.
\newblock Exponential {T}hurston maps and limits of quadratic differentials.
\newblock {\em J. Amer. Math. Soc.}, 22(1):77--117, 2009.

\bibitem[Hub06]{Hubbard_Book_1}
John~Hamal Hubbard.
\newblock {\em Teichm\"{u}ller theory and applications to geometry, topology, and dynamics. {V}ol. 1}.
\newblock Matrix Editions, Ithaca, NY, 2006.
\newblock Teichm\"{u}ller theory, With contributions by Adrien Douady, William Dunbar, Roland Roeder, Sylvain Bonnot, David Brown, Allen Hatcher, Chris Hruska and Sudeb Mitra, With forewords by William Thurston and Clifford Earle.

\bibitem[Hub16]{Hubbard_Book_2}
J.H. Hubbard.
\newblock {\em Teichm\"{u}ller theory and applications to geometry, topology, and dynamics, {V}olume~2: {S}urface homeomorphisms and rational functions}.
\newblock Matrix Editions, Ithaca, NY, 2016.

\bibitem[JB23]{MargalitWinarski}
Rebecca R.~Winarski James~Belk, Dan~Margalit.
\newblock {T}hurston's theorem and the {N}ielsen-{T}hurston classification via {T}eichmüller's theorem.
\newblock {\em Preprint arXiv:2309.06993}, 2023.

\bibitem[Koc13]{critically_finite_endomorphisms}
Sarah Koch.
\newblock Teichm\"{u}ller theory and critically finite endomorphisms.
\newblock {\em Advances in Mathematics}, 248:573--617, 2013.

\bibitem[KPS16]{Pullback_invariants}
Sarah Koch, Kevin~M. Pilgrim, and Nikita Selinger.
\newblock Pullback invariants of {T}hurston maps.
\newblock {\em Trans. Amer. Math. Soc.}, 368(7):4621--4655, 2016.

\bibitem[Lee12]{Lee}
John~M. Lee.
\newblock {\em Introduction to Smooth Manifolds}, volume 218 (Second ed.) of {\em Graduate Texts in Mathematics}.
\newblock New York London: Springer-Verlag, [2012] \copyright 2012.

\bibitem[LP20]{Stoilow}
Rami Luisto and Pekka Pankka.
\newblock Sto\"{\i}low's theorem revisited.
\newblock {\em Expo. Math.}, 38(3):303--318, 2020.

\bibitem[MPR24]{our_approx}
Malavika Mukundan, Nikolai Prochorov, and Bernhard Reinke.
\newblock Dynamical approximations of postsingularly finite entire maps.
\newblock {\em Preprint arXiv:2305.17793}, 2024.

\bibitem[Nag82]{non_isometric_disks}
Subhashis Nag.
\newblock Non-geodesic discs embedded in {T}eichmüller spaces.
\newblock {\em American Journal of Mathematics}, 104(2):399--408, 1982.

\bibitem[Par22]{park2}
Insung Park.
\newblock Julia sets with {A}hlfors-regular conformal dimension one.
\newblock {\em Preprint arXiv:2209.13384}, 2022.

\bibitem[Par23]{park}
Insung Park.
\newblock Levy and {T}hurston obstructions of finite subdivision rules.
\newblock {\em Ergodic Theory and Dynamical Systems}, pages 1--51, 12 2023.

\bibitem[Pil03]{Pilgrim_Comb}
K.M. Pilgrim.
\newblock {\em Combinations of complex dynamical systems}.
\newblock Lect. Notes Math. 1827. Springer, Berlin, 2003.

\bibitem[Pro24]{small_postsingular_set}
Nikolai Prochorov.
\newblock Finite and infinite degree thurston maps with a small postsingular set.
\newblock {\em Preprint arXiv:2410.01146}, 2024.

\bibitem[Sel13]{Selinger_Top_Obstr}
N.~Selinger.
\newblock Topological characterization of canonical {T}hurston obstructions.
\newblock {\em J. Mod. Dyn.}, 7(1):99--117, 2013.

\bibitem[Sha13]{sharland}
Thomas Sharland.
\newblock Constructing rational maps with cluster points using the mating operation.
\newblock {\em J. Lond. Math. Soc. (2)}, 87(1):87--110, 2013.

\bibitem[She22]{Shemyakov_Thesis}
Sergey Shemyakov.
\newblock {\em A topological characterization of certain postsingularly finite entire functions: transcendental dynamics and Thurston theory}.
\newblock PhD thesis, Universit\'{e} d'Aix-Marseille, 2022.
\newblock Available at https://www.theses.fr/2022AIXM0017.

\bibitem[SY15]{Selinger_Yampolski}
Nikita Selinger and Michael Yampolsky.
\newblock Constructive geometrization of {T}hurston maps and decidability of {T}hurston equivalence.
\newblock {\em Arnold Math. J.}, 1(4):361--402, 2015.

\end{thebibliography}

\end{document}